 \documentclass[smallextended]{interact}
 \usepackage{amsmath}
 \usepackage{enumerate}
 \usepackage{multirow}
 \usepackage{graphicx}
 \usepackage{amssymb}
 \usepackage{amsthm}
 \usepackage{epstopdf,geometry}
 
 \usepackage{booktabs}
 \usepackage{amsfonts}
 \usepackage{float}
 \usepackage{amsfonts}
 \usepackage{multirow}
 \usepackage{graphicx}
 \usepackage{subfigure}
 \newtheorem{example}{Example}[section]
 \newtheorem{assumption}{Assumption}
 \usepackage{algorithm} 
 \usepackage{algorithmic} 

 \usepackage{epstopdf}
 \usepackage[caption=false]{subfig}

 \usepackage[numbers,sort&compress]{natbib}
 \bibpunct[, ]{[}{]}{,}{n}{,}{,}
 
 \theoremstyle{plain}
 \newtheorem{theorem}{Theorem}[section]
 \newtheorem{lemma}[theorem]{Lemma}
 \newtheorem{corollary}[theorem]{Corollary}
 \theoremstyle{definition}
 \newtheorem{definition}[theorem]{Definition}
 \theoremstyle{remark}
 \newtheorem{remark}{Remark}

\begin{document}

\title{ The Smoothing Proximal Gradient Algorithm with Extrapolation for the Relaxation of $l_0$ Regularization Problem
}


\articletype{ARTICLE TEMPLATE}

\title{ The Smoothing Proximal Gradient Algorithm with Extrapolation for the Relaxation of $L_0$ Regularization Problem}

\author{
	\name{Jie Zhang\textsuperscript{a} and Xinmin Yang\textsuperscript{b} and Gaoxi Li\textsuperscript{c} and Ke Zhang\textsuperscript{d}\thanks{CONTACT  G.X.Li. Author. Email: ligaoxicn@126.com}}
	\affil{\textsuperscript{a} College of Mathematics, Sichuan University, Chengdu 610065, China\\
		\textsuperscript{b} School of Mathematics Science, Chongqing Normal University, Chongqing 401331, China\\
		\textsuperscript{c} School of Mathematics and Statistics, Chongqing Technology and Business University, Chongqing, China\\
		\textsuperscript{d}   School of Mathematics and Computer Science, Chongqing College of International Business and Economics, Chongqing, China
	}
}

\maketitle
\begin{abstract}
 In this paper, we consider the exact continuous relaxation model of $l_0$ regularization problem which was given by Bian and Chen (SIAM J. Numer. Anal 58(1): 858-883, 2020) and propose a smoothing proximal gradient algorithm with extrapolation (SPGE) for this kind of problem. We show that any accumulation point of the sequence generated by SPGE algorithm is a lifted stationary point of the relaxation model for a fairly general choice of extrapolation parameter.
 Moreover, it is shown that the whole sequence
 generated by SPGE algorithm converges to a lifted stationary point of the relaxation model.
 The convergence rate of $o(\frac{1}{k^{1-\sigma}})(\sigma\in(0,1))$  with respect to squared proximal residual is established.
   Finally, we conduct three important numerical experiments to illustrate the faster efficiency of the SPGE algorithm compared with the smoothing proximal gradient(SPG) algorithm proposed by Bian and Chen (SIAM J. Numer. Anal 58(1):858-883, 2020).

\end{abstract}
\section{Introduction}
\label{intro}

In this paper, we concentrate on the exact continuous relaxation model of $l_0$ regularization problem which was proposed in \cite{bianSmoothingProximalGradient2020}, i.e.,
\begin{align}
\mathop{\min}\limits_{x\in \mathcal{X}} \quad \mathcal{F}(x):=f(x)+\lambda \Phi(x),\label{1.2}
\end{align}
  where  $\mathcal{X}\!=\!\left\{x\in R^n\!:\! l\leq x\leq u\right\}$, $l, u\in \left\{0, \pm \infty \right\}^n$ with $l \leq 0 \leq u$ and $l<u$, $\Phi(x)=\sum_{i=1}^{n} \phi(x_i)$, $f:R^n\to [0,\infty)$ is convex, not necessarily smooth and has Lipschitz continuous constant $L_f$ on $\mathcal{X}$. $\lambda$ is a positive parameter keeping a trade-off between data fidelity and sparsity. The problem (\ref{1.2}) is an exact continuous relaxation of the $l_0$ regularization problem and the $l_0$  regularization problem is specifically expressed as
 \begin{align}
\mathop{\min}\limits_{x\in\mathcal{X}} \quad \mathcal{F}_{l_0}(x):=f(x)+\lambda {\Vert x\Vert}_0,\label{1.1}
\end{align}
where $\Vert\cdot\Vert_0$ denotes the so called $l_0$ norm. For a vector $x\in R^n$, $\Vert x\Vert_0= \vert\mathcal{A}(x)\vert$, where $\mathcal{A}(x)=\left\{i\in \left\{1,2,3,\cdots,n \right\}:x_i\ne 0\right\}$ and $\vert\mathcal{A}(x)\vert$ is the cardinality of the vector $x$. We say that $x\in R^n$ is sparse if $\vert \mathcal{A}(x)\vert\ll n$. Given a parameter $v>0$, the capped-$l_1$ function  $\phi$ is
 \begin{align}
\phi(t)=\mathop{\min\left\{1, |t|/v\right\}}\label{1.3}.
\end{align}
Moreover,  $\phi$ in (\ref{1.3}) is a DC function which can be represented in the following form,
\begin{align}
\phi(t) =\frac{\vert t\vert}{v}-\max\left\{\theta_1(t),\theta_2(t),\theta_3(t)\right\}\label{2.1},
\end{align}
with $\theta_1(t)=0, \,\theta_2(t)=t/v-1,\, \theta_3(t)=-t/v-1$
and $\mathcal{D}(t)=\{i=\{1,2,3\}: \theta_i(t)=\max\{\theta_1(t),\theta_2(t), \theta_3(t)\}\} $.

 In {\cite{bianSmoothingProximalGradient2020}, Bian and Chen proposed  SPG algorithm which combines the smoothing method and proximal gradient algorithm to solve the exact continuous relaxation model of the $l_0$ regularization problem. It is well known that the proximal gradient method is always an usual and efficient method to deal with the following problem 
 \begin{align}\label{1.1a}
  \min_{x\in R^n} {\mathcal F}(x):=f(x)+g(x),
  \end{align}
  where $f: R^n\to R$ is a smooth and possibly nonconvex function, $g:R^n\to (-\infty,+\infty]$ is a proper closed, nonsmooth and possibly nonconvex function. This kind of problem includes many optimization models deriving from various applications such as machine learning
  \cite{clarksonSublinearOptimizationMachine2010,jainNonconvexOptimizationMachine2017}, signal and image processing \cite{ghadimiAcceleratedGradientMethods2016}, compressed sensing \cite{donohoCompressedSensing2006a} and so on. The proximal mapping  \cite{parikhProximalAlgorithms2014b} of $tg$ is often assumed to be easily computed for all $t>0, u\in R^n$  and is defined as follows
  \[   prox_{tg}(u): = { \mathop{argmin}\limits}\left\{g(x)+\frac{1}{2t}{\Vert x-u\Vert}^2 \right\}.
  \]

 Since the low cost of the proximal gradient method, a great many of researchers  devote to it and aim to improve its efficiency. One common and efficient technique is to perform extrapolation,
 where momentum terms involving the previous iterations are added to the current iteration. One of the accelerated versions of the proximal gradient method was introduced by Ochs et al. \cite{ochsIPianoInertialProximal2014b}. They proposed to incorporate the inertial force which is also called extrapolation item into   proximal step and studied the inertial proximal gradient algorithm for nonconvex optimization (iPiano). It can be seen as a nonsmooth split version of the Heavy-ball method from Polyak  \cite{polyakMethodsSpeedingConvergence1964a}, the update system of the algorithm is
\begin{align}
 y^k=&x^k+  \beta_k(x^k-x^{k-1}), \notag\\
 x^{k+1}=&{  prox_{{t}_k g}}(y_k-{t}_k\nabla f(x_k)),\notag
\end{align}
where $\beta_k$, $t_k$ are the extrapolation coefficient and stepsize respectively  satisfying some rules and $prox$ is the proximal operator.
 The convergence rate $O(\frac{1}{k})$ with respect to squared proximal residual is obtained  and global convergence is exhibited based on the Kurdyka-{\L}ojasiewicz (K-L) property.
Another acceleration method adds the inertial force  not only  to   proximal step but
also to   gradient step which was raised by Nesterov
\cite{nesterovIntroductoryLecturesConvex2004,bFirstOrderMethodsOptimization2017},
  the general framework is as follows:
\begin{align*}
  y^k=&  x^k+\beta_k(x^k-x^{k-1}),  \notag\\
  x^{k+1}= &{  prox_{{t}_k g}}(y_k-{t}_k\nabla f(y_k)), \notag
 \end{align*}
where $\beta_k\in[0,1]$ is the extrapolation coefficient,  $t_k$ is the stepsize depending on the Lipschitz constant of $\nabla f(x)$.
 Based on the Nesterov's idea, Beck and Teboulle \cite{beckFastIterativeShrinkageThresholding2009a} developed a specific recurrence relation and exhibited  the remarkable Fast Iterative Shrinkage Thresholding Algorithm (FISTA) for the convex case and the convergence rate $O(\frac{1}{k^2})$ of the objective function is obtained. Besides, some variants of FISTA are established by choosing appropriate extrapolation parameters, we refer the reader to
 \cite{beckFastIterativeShrinkageThresholding2009a,chambolleConvergenceIteratesFast,liangFasterFISTA} and their references therein. We also observe that Wen et al. \cite{wenLinearConvergenceProximal2017} established the proximal gradient algorithm with extrapolation (PGe) for the case with gradient Lipschitz continuously differentiable
but nonconvex loss function. The results about the iterative sequence and objective function sequence converging R-linearly to a stationary point of the problem  are  obtained under the error bound condition \cite{Luo and Tseng}.

From the literatures mentioned above, all of the efficient accelerated algorithms require a restrictive  assumption that $f$ is a smooth function. Actually, the function $f$ in problem (\ref{1.2}) is always nonsmooth in many practical applications, such as the $l_1$ loss function \[f(x)=\frac{1}{m}\Vert Ax-b\Vert_1,\] and the censored regression problem
\begin{align*}
f(x)
=\frac{1}{pm} \sum_{i=1}^{m} \left\{\vert\max \left\{  A_ix-c_i,0 \right\}-b_i   \vert  \right\}^p,
\end{align*}
where $p\in[1,2]$, $A^T_i\in R^n$ and $c_i,b_i\in R, i=1,2,\cdots,m.$ Few researchers have focused on the accelerated algorithm  when $f$ is nonsmooth, for this, the nature consideration is to smooth it by virtue of the smoothing method which is a well recognized technique for the nonsmooth optimization problem. One can solve it by an appropriate smooth approximation and solve a sequence of smooth problems  gradually approaching to the original problem.
Smooth approximation method for optimization problems involves many kinds of problems,
one can refer to the literatures
\cite{chenSmoothingMethodsNonsmooth2012b,weibianSmoothingNeuralNetwork2012,chenSmoothingMethodMathematical2004,chenSmoothingMethodsSemismooth2000,huSmoothingApproachNash2012,nesterovSmoothingTechniqueIts2004,polyakNonlinearRescalingVs2002a,qiNewLookSmoothing2000,zhangSmoothingProjectedGradient2009} and their references therein.
It needs to mention that Nesterov \cite{nesterovSmoothMinimizationNonsmooth2005c} proposed an accelerated gradient algorithm based on a smoothing technique for problem (\ref{1.1a}) with $g(x)=0$ and improved the efficiency estimate from $O(\frac{1}{{\epsilon}^2})$ to
$O(\frac{1}{{\epsilon}})$ . We also observe that \cite{beckSmoothingFirstOrder2012a}  proposed a unifying framework that combines smoothing approximation based on the extension of the Moreau-infimal convolution technique with fast first order algorithm and proved its efficiency estimate.

Recently, we observe that Wu et al. \cite{Wu F Bian W} put forward a smoothing
fast iterative hard thresholding (SFIHT) algorithm solving $l_0$ regularized nonsmooth convex problem (\ref{1.1}) which together the smoothing approximations, extrapolation techniques
with iterative hard thresholding methods.  The extrapolation coefficient can be
achieved $\sup_k \beta_k= 1 $ in the SFIHT algorithm. They discussed the convergence of the SFIHT algorithm and obtained the convergence rate $O(\frac{lnk}{k})$ of the objective function.

To the best of our knowledge, there no scholars have studied the accelerated algorithm for the relaxation model (\ref{1.2}) of the $l_0$ regularized problem when $f$ is nonsmooth.
Benefited from acceleration technique and smoothing methods mentioned above,
we propose a SPGE algorithm for the relaxation model (\ref{1.2}) of (\ref{1.1}). Since the extrapolation item is added, the convergence can't be easily got directly. But we can discuss its convergence by constructing the appropriate auxiliary sequence.
We mainly analyze the convergence of the SPGE algorithm when the extrapolation term satisfying
$ \beta_k\in[0,\sqrt{(1-a\frac{\mu_{k+1}}{\mu_k})\frac{\mu_{k+1}}{\mu_k}}]$ with $0<a\ll 1$ which is a fairly general choice of extrapolation parameters, in particular,  the parameters can be chosen as in s-FISTA proposed by Bian \cite{WBA} with fixed restart \cite{Adaptiverestart}. Though the range of the extrapolation term slightly smaller compare with \cite{Wu F Bian W} , the better theoretical  results can be obtained. The result that  the sequence $\{ \Vert x^k-x^{k-1}\Vert  \}$ is summable and the whole sequence $\{x^k\}$ is convergent.  Further, the convergence rate with respect to proximal residual is obtained. 

The paper is organised as follows. In Section 2, we present some brief and useful preliminaries derived
from \cite{bianSmoothingProximalGradient2020}. In Section 3, we  propose the SPGE algorithm and analysis its the convergence behaviour.   It is proved that any accumulation point of a subsequence generated by SPGE algorithm is a lifted stationary point of the continuous relaxation model. Further, the whole sequence $\{x^k\}$ generated by SPGE algorithm convergences to a lifted stationary point of the relaxition model (\ref{1.2}).
 Besides, the convergence rate as regards the squared proximal residual is obtained.
In Section 4, we show that the efficiency of the proposed SPGE algorithm. Finally, we give some conclusions in Section 5.

\textbf{Notations.} Given a vector $x\in R^n$,  $ \Vert x\Vert_2$ is replaced by ${\Vert x\Vert}$. Denote $\mathbb{N}=\left\{0,1,2,\cdots \right\}$, $\mathbb{D}^n=\left\{d\in R^n: d_i\in \left\{1,2,3 \right\},i\\=1,2,\cdots,n\right\}$. For a nonempty, closed, convex set $\mathcal{X}\subseteq R^n$, $N_{\mathcal{X}}(x)$ represents the normal cone to $\mathcal{X}$ at $x\in\mathcal{X}$. $\textbf{1}_n\in R^n$ be the all ones vector and $e_i\in R^n$ be the $i$ column of the n-dimensional identity matrix. Denote $\partial f(x)$ be a Clarke subgradient of $f$ at $x\in R^n$ where $f(x): R^n\to R$ is a locally Lipschitz continuous function.

\section{Preliminaries and main results}\label{sec2}
 In this section, we first briefly recall the definition which will be useful for the analysis in this paper.

\begin{definition}\label{def2.1}(lifted stationary point
 {\cite{bianSmoothingProximalGradient2020,pangComputingBStationaryPoints2017}})
We say that $x\in \mathcal{X}$ is a lifted stationary point of (\ref{1.2}) if there exist $d_i\in \mathcal{D}(x_i)$ for $i=1,2,\cdots,n$ such that
\begin{align}
\lambda\sum_{i=1}^{n}\theta'_{d_i}(x_i)e_i\in\partial f(x)+\frac{\lambda}{v}\partial(\sum_{i=1}^{n}\vert x_i\vert)+N_{\mathcal{X}}(x).\label{2.2}
\end{align}
\end{definition}

 We make the following assumption in the convergence analysis.
 \begin{assumption} \label{assumption 2}
The parameter $v$ in (\ref{1.3}) satisfies $0<v< \overline v:=\frac{\lambda}{L_f}$.
\end{assumption}
\begin{assumption} \label{assumption 3}
 $\mathcal{F}$ in (\ref{1.2})(or $\mathcal{F}_{l_0}$ in (\ref{1.1}) ) is level bounded on $\mathcal{X}$.
\end{assumption}

Next we present the following lemmas which characterise the lifted stationary point of (\ref{1.2}) under the Assumption \ref{assumption 2},\ref{assumption 3}  and show the relationship between the lifted stationary point, local minimizers and global minimizers of (\ref{1.2}),(\ref{1.1}).
\begin{lemma}\label{lemma2.2}\rm{(\cite{bianSmoothingProximalGradient2020})}
If $\overline x$ is a lifted stationary point of (\ref{1.2}), then the vector
 $d^{\overline x}=(d_{1}^{\overline x},\cdots,d_{n}^{\overline x})\in \prod_{i=1}^{n}\mathcal{D}({\overline x}_i)$ satisfying (\ref{2.2}) is unique. In particular, for $i=1,2,\cdots, n$,
  \begin{equation}\label{2.3}
d_{i}^{\overline x}= \begin{cases}
1 &\text{ if } \vert{\overline x}_i\vert <v,\\
2 &\text{ if } {\overline x}_i\ge v ,\\
3 &\text{ if } {\overline x}_i\leq -v.
\end{cases}
\end{equation}
\end{lemma}

\begin{lemma}\label{lemma2.3}\rm(\cite{bianSmoothingProximalGradient2020})
If $\overline x$ is a lifted stationary point of (\ref{1.2}), then it has the lower bound property, i.e.,
if ${\overline x}_i\in(-v,v)$, then ${\overline x}_i=0, \forall\, i=1,\cdots,n$.
\end{lemma}

\begin{lemma}\label{lemma2.4}\rm(\cite{bianSmoothingProximalGradient2020})
If $\overline x$ is a lifted stationary point of (\ref{1.2}), then it is a local minimizer of (\ref{1.1}) and the objective functions have the same value at $\overline x$.
\end{lemma}
\begin{lemma}\label{lemma2.4a}\rm(\cite{bianSmoothingProximalGradient2020})
 $\overline x$ is a lifted stationary point of (\ref{1.2}) and satisfies $\vert x_i\vert>v $ when $x_i\ne 0, \forall i$ if and only if  it is a local minimizer of (\ref{1.2}).
\end{lemma}

\begin{lemma}\label{lemma2.5}\rm(\cite{bianSmoothingProximalGradient2020})
 $\overline x$ is a global minimizer of (\ref{1.2}) if and only if it is a global minimizer of (\ref{1.1}) and the objective functions have the same value at $\overline x$.
\end{lemma}
\section{The smoothing proximal gradient algorithm with extrapolation and its convergence analysis}\label{sec3}
    Combining the smoothing method and the proximal gradient method with extrapolation, we propose the SPGE     algorithm and discuss its convergence in this section. For convenience, we first exhibit the definition of smoothing function deriving from \cite{bianSmoothingProximalGradient2020}.
\begin{definition}\label{def3.1}
We call $\widetilde f:R^n\times [0,\overline \mu]\to{R}$ with $\overline \mu>0$ a smoothing function of the function $f$ in (\ref{1.2})   if $\widetilde f(\cdot,\mu)$ is continuous differentiable in $R^n$ for any fixed $\mu>0$ and satisfies the following conditions:
\begin{enumerate}[(i)]
  \item  $ {\lim_{z\to x,\mu\downarrow 0}}\widetilde f(z,\mu)=f(x) , \quad \forall x\in\mathcal{X}$;
   \item  $\widetilde f(x,\mu)$ is convex with respect to $x$ in $\mathcal{X}$ and for any fixted $\mu>0$;
     \item $ \{\lim_{z\to x,\mu\downarrow 0}\nabla_z\widetilde f(z,\mu)\}\subseteq\partial f(x) , \quad \forall x\in\mathcal X$;
  \item there exists a positive constant $\kappa$ such that
\[
\vert \widetilde f(x,\mu_2)-\widetilde f(x,\mu_1)\vert\leq\kappa\vert \mu_1-\mu_2\vert,\quad \forall x\in\mathcal{X},  \mu_1,\mu_2\in[0,\overline\mu];
\]
especially,
$
\vert \widetilde f(x,\mu)- f(x)\vert\leq\kappa  \mu ,\quad \forall x\in\mathcal{X},  \mu\in (0,\overline\mu];
$
  \item for any $\mu\in(0,\overline\mu]$, there exists a constant $\widetilde L>0$ such that $\nabla_x\widetilde f(\cdot,\mu)$ is Lipschitz continuous on $3\mathcal {X}$ with Lipschitz constant $\widetilde L{\mu}^{-1}$.
\end{enumerate}
\end{definition}
\begin{remark}
	The definition in (v) says that $\nabla_x\widetilde f(\cdot,\mu)$ is Lipschitz continuous on $3\mathcal{X}$ since the extrapolation point $y^k =x^k+\beta_k(x^k-x^{k-1})$ in the following algorithm may not  belong to $\mathcal{X}$ and we need to restrict it on $3\mathcal{X}$.
	\end{remark}
   For a given $d=(d_1,d_2,\cdots,d_n)^T\in \mathbb{D}^n$, let
\begin{align}\label{3.1}
\Phi^d(x):=\sum_{i=1}^{n}(\frac{\vert x_i\vert}{v}-\theta_{d_i}(x_i)),
\end{align}
then
 \begin{align}\label{3.2}
 \Phi(x)=\mathop{\min}_{d\in\mathbb{D}^n}\Phi^d(x), \quad \forall x\in \mathcal{X}.
 \end{align}
 And especially  for a fixed $\overline x\in \mathcal{X}$,
  $\Phi(\overline x)=\Phi^{d^{\overline x}}(\overline x)$, where $d^{\overline x}$ is defined in (\ref{2.3}).
\[  \mathcal{\widetilde F}^d(x,\mu):= \widetilde f(x,\mu)+ \lambda\Phi^d(x),\,
 \mathcal{\widetilde F}(x,\mu):= \widetilde f(x,\mu)+ \lambda\Phi(x),
\]
where $\widetilde f$ is a smoothing function of $f$. For any fixed $\mu>0$, and $d\in \mathbb{D}^n$, $ \mathcal{\widetilde F}^d(x,\mu)$ is convex function.\\
By (\ref{3.2}),
it holds that
\begin{align}\label{3.3}
\mathcal{\widetilde F}^d (x,\mu)\ge \mathcal{\widetilde F}(x,\mu)\quad \forall d\in\mathbb D^n, x\in\mathcal{X},\mu\in(0,\overline\mu].
\end{align}
 In the following, we focus on the constrained convex optimization problem with a given smoothing parameter $\mu>0$, $d\in \mathbb{D}^n$:
\begin{align}\label{3.4}
 \mathcal{\widetilde F}^d(x,\mu):= \widetilde f(x,\mu)+ \lambda\Phi^d(x) .
\end{align}
Let $d\in\mathbb{D}^n, w\in R^n$ and $\tau>0$. Due to the fact that penalty item of  (\ref{3.4}) is piecewise linear,
  the proximal operator of $\tau\Phi^d$ on $\mathcal{X}$
 \[{ prox_{\tau\Phi^d}(w)}={\rm\mathop{argmin}}_{\mathcal X}\left\{\tau\Phi^d(x)+\frac{1}{2}{\Vert x-w \Vert}^2\right\}\]
    has a closed form solution $\hat x=(\hat x_1, \cdots, \hat x_n)$ with
  $\hat x_i=\mathop{\min}\left\{\max \left\{l_i,h_i\right\},u_i\right\} $,
 $\forall i=1,2,\cdots,n,$
 where  \begin{equation}\label{3.6}
h_{i}= \begin{cases}
0 &\text{ if } \vert\widetilde w_i\vert\leq\tau/v,\\
\widetilde w_i-\tau/v,&\text{ if } \widetilde w_i > \tau/v ,\\
\widetilde w_i+\tau/v,&\text{ if } \widetilde w_i<-\tau/v,
\end{cases}
\end{equation}
where
 \begin{equation*}\label{3.6a}
\widetilde w_i= \begin{cases}
w_i, &\text{ if }d_i=1, \\
w_i+\tau/v,&\text{ if }  d_i=2 ,\\
w_i-\tau/v,&\text{ if } d_i=3.
\end{cases}
\end{equation*}

 Based on the proximal gradient method, we consider the approximation $Q_{d^k}(x,y^k,\mu_k)$ of $\widetilde {\mathcal F}^d(\cdot,\mu_k)$ and the given point $y^k$  is involved  two previous iterations,
 specifically as follows:
\begin{equation*}\label{3.7}
\begin{cases}
	y^k=x^k+\beta_{k-1}(x^k-x^{k-1}),\\
Q_{d^k}(x,y^k,\mu_k)\!=\!\widetilde f(y^k,\mu_k)\!+\!\langle x-y^k, \nabla \widetilde f(y^k,\mu_k)\rangle \!+\!\frac{1}{2}L{\mu^{-1}_k}{\Vert x-y^k \Vert}^2+\lambda{{\Phi}^{d^k}}(x),

\end{cases}
\end{equation*}
where $\beta_{k-1}$ is an extrapolation parameter, $L$ is any constant satisfing $L\ge \widetilde L$.  It is observed that the problem
 $ \mathop{\min}\limits_x{Q_{d^k}(x,y^k,\mu_k)}$ has a unique minimizer $\hat x$ since
  $ Q_{d^k}(x,y^k,\mu_k)$ is a strongly convex function with respect to $x$ for any fixed $d^k,\mu_k, y^k$, where $d^k =d^{x^k},\mu_k$ is the smoothing parameter. The minimizer can be obtained by (\ref{3.6}) with $\tau=\lambda L^{-1}\mu_k$ and $w^k=y^k-{L}^{-1}\mu_k\nabla\widetilde f(y^k,\mu_k)$.\\
\begin{algorithm}[hptb]
\caption{SPGE algorithm} 
\label{alg1} 
\begin{algorithmic}[1] 
\STATE \textbf{Data:} $\sigma\in(0,1)$, $\alpha>0$,

\textbf{Initialization:} $x^{-1}=x^{0}\in \mathcal{X}$, ${\mu}_{-1}={\mu}_0\in(0,\overline\mu],$ $L>\widetilde L$. Set $k=0$.
\WHILE{a termination criterion is not met}
\STATE
 \textbf{step 1:} Choose $\beta_{k-1}\in[0,\sqrt{\frac{\mu_{k}}{\mu_{k-1}}})$,  and let $d^k=d^{x^k}$, where $d_i^{x^k}$ is defined in (\ref{2.3}).\\
   \textbf{step 2:}
Compute
\begin{equation}\label{3.7a}
\begin{cases}
	y^k=x^k+\beta_{k-1}(x^k-x^{k-1}),\\
  x^{k+1}= {\rm\mathop argmin}_{x\in \mathcal{X}} {Q_{d^k}(x,y^k,\mu_k)};

\end{cases}
\end{equation}

   \textbf{step 3:}
If
\begin{align}\label{3.9}
H (x^{k+1},x^k, \mu_k)+\kappa\mu_{k}-(H(x^{k},x^{k-1},\mu_{k-1})+\kappa\mu_{k-1})\leq -\alpha {\mu_k}^2,
\end{align}
set $\mu_{k+1}=\mu_k$, otherwise, set
\begin{align}\label{3.10}
\mu_{k+1}=\frac{\mu_0}{(k+1)^\sigma}.
\end{align}
set $k:=k+1$, and return to Step 1.
\ENDWHILE
\end{algorithmic}
\end{algorithm}
\begin{remark}
	It is need to be emphasized that the following convergence analysis  is discussed under the extrapolation coefficient $\beta_k$ condtition which is 
	\begin{align}\label{100}
		0\leq\beta^2_k\leq (1-a\frac{\mu_{k+1}}{\mu_k})\frac{\mu_{k+1}}{\mu_k},\quad with\,\, 0<a\ll 1.
	\end{align}
\end{remark}
\begin{remark}
	The sequence $\{H(x^{k+1},x^k,\mu_k)\}$ in the SPGE algorithm framework is specifically defined as in (\ref{3.11a}).
\end{remark}
\begin{remark}
From the above algorithm framework, it can be found that the subproblem (\ref{3.7a}) boils down to the SPG algorithm when $\beta_k=0$, therefore, the SPG algorithm is a special case of the SPGE algorithm. Besides,  when $\beta_k$ is chosen as (\ref{100}), that is $sup \beta_k\leq\sqrt{ 1-a}<1$ with $ 0<a\ll 1$, it 
 is more general and contains one special extrapolation coefficient deriving from the smoothing fast iterative shrinkage thresholding algorithm (s-FISTA) which was proposed in \cite{WBA} with fixed restart or adaptive restart system \cite{Adaptiverestart} for solving subproblem (\ref{3.7a}). That is, one reset the initial point $t_{-1}=1$, ${\mu}_{-1}={\mu}_0 $ after fixed iterations or restart the algorithm when $\langle y^{k-1}-x^k, x^k-x^{k-1}\rangle>0$. The extrapolation parameters in experiments of Section 4 is chosen as in the way of the fixed restart.
 The extrapolation coefficient in s-FISTA is defined as follows, for any $k\ge0$
\begin{align}\label{A}
t_{k }=\frac{1+\sqrt{1+4\frac{\mu_{k-1}}{\mu_{k}}{t^2_{k-1}}}}{2},  \quad  \beta_{k-1}=\frac{ t_{k-1}-1}{t_{k}},
   \end{align}
where $t_{-1}=1$.
\end{remark}

%
%

In what follows we analyze the basic convergence behaviors of the iterative sequence generated by the SPGE algorithm. Though it gets much more difficult to prove directly, added the extrapolation, we can develop the convergence analysis by constructing the following auxiliary sequence. Suppose $\{x^k\}, \{y^k\},\{\mu_k\}$ are the sequences generated by SPGE algorithm, define
\begin{align}\label{3.11a}
H(x^{k+1}, x^k,\mu_{k}):=\widetilde {\mathcal F}(x^{k+1},\mu_{k})+ \tau_k {\Vert x^{k+1}-x^k\Vert }^2,
\end{align}
where $\tau_k=\frac{L{\mu^{-1}_k}}{4}+\frac{L{\beta^2_{k}}\mu^{-1}_{k+1}}{4}$.

 It needs to be emphasised that the convergence analysis of the SPGE algorithm is heavily based on the auxiliary sequence.
\begin{lemma}\label{lemma3.2}(well defined of the SPGE algorithm)
The proposed SPGE algorithm is well defined, and the sequences $\left\{x^k\right\}$,   $\left\{\mu_k\right\}$ generated by the SPGE have the following properties.
\begin{enumerate}[(i)]
   \rm \item $\left\{x^k\right\}\subseteq \mathcal{X}$.
  \rm \item   there are infinite elements in $\mathcal{N}^s$ and $\lim_{k\to +\infty}\mu_k=0$, where
$\mathcal{N}^s=\left\{k\in\mathbb{N}: \mu_{k+1}\ne\mu_{k}\right\}.$
\end{enumerate}
 \end{lemma}
\begin{proof}
The proof is similar to Lemma 3.2 in \cite{bianSmoothingProximalGradient2020},  the results can be easily obtained.\qed
\end{proof}
Now we show that the auxiliary  sequence $\{H(x^{k+1},x^k, \mu_k)+\kappa \mu_k  \}$
is nonincreasing when $\beta_k\in[0,\sqrt{\frac{\mu_{k+1}}{\mu_k}})$.


\begin{lemma}\label{lemma3.3}
For any $k\in \mathbb{N}$, if $\beta_k\in[0,\sqrt{\frac{\mu_{k+1}}{\mu_{k}}})$,
then
\begin{align}\label{3.11}
& H (x^{k+1},x^k,\mu_k)+\kappa \mu_k-(H(x^{k}, x^{k-1},\mu_{k-1})+\kappa\mu_{k-1}) \notag\\
\leq &   (\tau_k-\frac{1}{2}L\mu^{-1}_k) {\Vert x^{k+1}-x^k\Vert}^2-(\tau_{k-1}-\frac{1}{2}L\mu^{-1}_k\beta^2_{k-1}) {\Vert x^{k}-x^{k-1}\Vert}^2,
\end{align}
and the sequence $\{H (x^{k+1},x^k,\mu_k)+\kappa \mu_k\}$ is a nonincreasing sequence.
\end{lemma}
\begin{proof}
From the strong convexity of $Q_{d^k}(x,y^k,\mu_k)$ with the strong convex coefficient $L\mu^{-1}_k$, we have
 \begin{align}\label{3.13}
 \lambda{{\Phi}^{d^k}}(x^{k+1})\leq& \lambda{{\Phi}^{d^k}}(x)+\langle x-x^{k+1}, \nabla \widetilde f(y^k,\mu_k)\rangle+\frac{1}{2} L{\mu^{-1}_k}{\Vert x-y^k \Vert}^2 \notag\\ &-\frac{1}{2}L{\mu^{-1}_k}{\Vert x^{k+1}-y^k \Vert}^2
 -\frac{1}{2}L{\mu^{-1}_k}{\Vert x-x^{k+1} \Vert}^2,
  \end{align}
Since $\nabla_x\widetilde f(\cdot,\mu)$  has Lipschitz constant $\widetilde L{\mu_k}^{-1}$ from  the Definition \ref{def3.1}(v), it follows that
\begin{align}\label{3.14}
\widetilde f(x^{k+1},\mu_k)\leq \widetilde f(y^{k},\mu_k)+\langle x^{k+1}-y^k, \nabla \widetilde f(y^k,\mu_k)\rangle+\frac{1}{2}\widetilde L{\mu^{-1}_k}{\Vert x^{k+1}-y^k \Vert}^2.
\end{align}
Combining (\ref{3.13}) with (\ref{3.14}), we have
      \begin{align*}
\widetilde {\mathcal F}^{d^k}(x^{k+1},\mu_k)&\leq \lambda{{\Phi}^{d^k}}(x)\!+\!
 \widetilde f(y^{k},\mu_k)\!+\! \langle x-y^k, \nabla \widetilde f(y^k,\mu_k)\rangle\!+\! \frac{1}{2}L{\mu^{-1}_k}{\Vert x\!-\!y^k \Vert}^2 \notag\\
 &\quad - \frac{1}{2}L{\mu^{-1}_k}{\Vert x-x^{k+1} \Vert}^2.
 \end{align*}
Using the convexity of $\widetilde f(x, \mu_k)$ with respect to $x$, the above inequality can be rewritten as
\begin{align}
\widetilde  {\mathcal F}^{d^k}(x^{k+1},\mu_k)
&\leq \widetilde  {\mathcal F}^{d^k}(x,\mu_k)+
\frac{1}{2}L{\mu^{-1}_k}{\Vert x-y^k \Vert}^2-\frac{1}{2}L{\mu^{-1}_k}{\Vert x-x^{k+1} \Vert}^2,
\end{align}

Setting $x=x^k$, $d^k=d^{x^k}$, according to $ \widetilde {\mathcal F}^{d^k}(x^{k+1},\mu_k)\ge \widetilde {\mathcal F}(x^{k+1},\mu_k)$, we have
\begin{align}\label{3.15}
 \widetilde  {\mathcal F} (x^{k+1},\mu_k)
&\leq \widetilde  {\mathcal F}^{d^k}(x^k,\mu_k)+
\frac{1}{2}L{\mu^{-1}_k}{\Vert x^k-y^k \Vert}^2-\frac{1}{2}L{\mu^{-1}_k}{\Vert x^k-x^{k+1} \Vert}^2.
\end{align}
Together the Definition \ref{def3.1}(iv) with the above inequality, we deduce
  \begin{align*}
 \widetilde  {\mathcal F} (x^{k+1},\mu_k)+\kappa\mu_k
&\leq \widetilde  {\mathcal F} (x^k,\mu_{k-1})+ \kappa\mu_{k-1}+
\frac{1}{2}L{\mu_k}^{-1}{\Vert x^k-y^k \Vert}^2 \notag\\
&\quad-\frac{1}{2}L{\mu^{-1}_k}{\Vert x^k-x^{k+1} \Vert}^2,
\end{align*}
 Then it has that
 \begin{align*}
&\widetilde {\mathcal F}(x^{k+1},\mu_k)+\kappa\mu_k+ \tau_k {\Vert x^k-x^{k+1} \Vert}^2\notag\\
\leq & \widetilde  {\mathcal F} (x^k,\mu_{k-1})+ \kappa\mu_{k-1}+\tau_{k-1} {\Vert x^{k-1}-x^{k} \Vert}^2\notag\\
\,&+(\tau_k-\frac{1}{2}L{\mu^{-1}_k}){\Vert x^k-x^{k+1} \Vert}^2-(\tau_{k-1}-\frac{1}{2}L{\mu^{-1}_k}\beta^2_{k-1}){\Vert x^k-x^{k-1}\Vert}^2.
\end{align*}
By the definition of $\tau_k$ from (\ref{3.11a}) and $\beta_k\in[0,\sqrt{\frac{\mu_{k+1}}{\mu_{k}}})$, we deduce that
\[ \tau_k-\frac{1}{2}L{\mu^{-1}_k}<0,\quad \tau_{k-1}-\frac{1}{2}L{\mu^{-1}_k}\beta^2_{k-1}>0,   \]
 and then the sequence $\{H (x^{k+1},x^k,\mu_k)+\kappa \mu_k\}$ is a nonincreasing sequence.\\
This completes the proof.\qed
\end{proof}

\begin{corollary}\label{corollary 1}
For any $k\in \mathbb{N}$, $\beta_k\in[0,\sqrt{\frac{\mu_{k+1}}{\mu_{k}}})$,  $\beta_k$ is chosen as (\ref{100}),  then there exists $\zeta\in R$ satisfying
 \[\zeta: =\mathop{\lim_{k\to+\infty}}H (x^{k+1},x^k,\mu_k)+\kappa\mu_{k},    \]
and
\begin{align}\label{3.16}
\mathop{\lim_{k\to+\infty}}H(x^{k+1},x^k, \mu_k)+\kappa\mu_{k}
& =\mathop{\lim_{k\to+\infty}}H(x^{k+1},x^k, \mu_k)\\ \notag
&= \mathop{\lim_{k\to+\infty}}\widetilde{\mathcal F}(x^{k},\mu_{k-1})=\mathop{\lim_{k\to+\infty}} {\mathcal F}(x^{k})=\zeta.
 \end{align}
Moreover, for any $k\in \mathbb{N}$, the sequence $\left\{ x^k\right\}$ is bounded.
\end{corollary}

\begin{proof}
  From (\ref{3.11a}), Definition \ref{def3.1}(iv) and Lemma \ref{lemma2.5}, we have
   \begin{align*}
   {H(x^{k+1},x^k, \mu_k)+\kappa \mu_k }\ge \widetilde {\mathcal{F}}(x^{k+1}, \mu_k)+\kappa \mu_k \ge  {\mathcal{F}}(x^{k+1}) &\ge\mathop{min_{x\in\mathcal{X}}}\mathcal{F}(x) \notag\\
   &=\mathop{min_{x\in\mathcal{X}}}\mathcal{F}_{l_0}(x).
   \end{align*}
   Together the lower bounded of the sequence $\left\{H(x^{k+1},x^k ,\mu_k)+\kappa \mu_k \right\}$ with its nonincreasing, then it is convergent and there exists $\zeta\in R$ satisfying
 \[\zeta: =\mathop{\lim_{k\to+\infty}}H (x^{k+1},x^k ,\mu_k)+\kappa\mu_{k}.   \]
     By virtue of $\mathop{\lim_{k\to+\infty} \mu_k}=0$ and  Definition \ref{def3.1}(i), we have
 \begin{align} \label{3.16a}
 \mathop{\lim_{k\to+\infty}}H(x^{k+1},x^k ,\mu_k)+\kappa\mu_{k}
=\mathop{\lim_{k\to+\infty}}H(x^{k+1},x^k ,\mu_k)=\zeta,
 \end{align}
 Next, we are ready to prove $\mathop{\lim_{k\to+\infty}}\widetilde{\mathcal F}(x^{k},\mu_{k-1})=\zeta.$  From (\ref{3.11}),  we deduce
 \begin{align}\label{3.19a}
  &\vert \widetilde {\mathcal F}(x^{k},\mu_{k-1})-\zeta\vert\notag\\
=&\vert H (x^{k},x^{k-1}, \mu_{k-1})-\zeta
 -\tau_{k-1} {\Vert x^k-x^{k-1} \Vert}^2 \vert\notag\\
\leq &\vert H(x^{k},x^{k-1} ,\mu_{k-1})-\zeta\vert
 + \frac{\tau_{k-1}}{\tau_{k-1}-\frac{1}{2}L{\mu^{-1}_k}\beta^2_{k-1}}( H (x^{k},x^{k-1}, \mu_{k-1})\notag\\
\quad &+\kappa \mu_{k-1}-(H(x^{k+1},x^{k},\mu_{k})+\kappa\mu_{k}))\notag\\
  \leq &\vert H(x^{k},x^{k-1}, \mu_{k-1})-\zeta\vert +\frac{1+\beta^2_{k-1}\frac{\mu_{k-1}}{\mu_k}}{1-\beta^2_{k-1}\frac{\mu_{k-1}}{\mu_k}} (H(x^{k},x^{k-1},\mu_{k-1})-\zeta)\notag\\
  \leq & \vert H(x^{k},x^{k-1}, \mu_{k-1})-\zeta\vert + \frac{2-a}{a}(H(x^{k},x^{k-1},\mu_{k-1})-\zeta),
 \end{align}
where the first inequality uses the triangle inequality and lemma \ref{lemma3.3}, the second one derives from the nonincreasing of $\left\{H (x^{k+1},x^k ,\mu_k)+\kappa \mu_k \right\}$ and
 $ H(x^{k},x^{k-1}, \\
 \mu_{k-1})+\kappa \mu_{k-1} \ge\zeta$, the last one is due to
$\beta^2_{k-1}{\frac{\mu_{k-1}}{\mu_{k}}}\leq 1-a\frac{\mu_{k-1}}{\mu_{k}}\leq 1-a$ and $\frac{1}{1-\beta^2_{k-1}\frac{\mu_{k-1}}{\mu_k}}\leq\frac{1}{a}$.
When $k$ is sufficiently large, the right of the inequality (\ref{3.19a}) tends to 0, thus
\[
 \mathop{\lim_{k\to+\infty}} \widetilde{\mathcal F} (x^{k},\mu_{k-1})=\zeta.
\]
and  $\mathop{\lim_{k\to+\infty}} \widetilde{\mathcal F}(x^k)=\mathop{\lim_{k\to+\infty}} \widetilde{\mathcal F} (x^{k},\mu_{k-1})=\zeta$ from Definition \ref{def3.1}(i).
Using the nonincreasing property of the sequence
 $\left\{ H(x^{k+1},x^k ,\mu_k)+\kappa\mu_{k}  \right\}$, we have
\begin{align*}
   {\mathcal F}(x^{k+1}) \leq \widetilde {\mathcal F}(x^{k+1}, \mu_k)+\kappa \mu_k \leq {H (x^{k+1},x^k,\mu_k)+\kappa \mu_k}&\leq H (x^{1},x^0,\mu_0)+\kappa \mu_0 \notag\\
    &<+\infty.
    \end{align*}
Since ${\mathcal F}$  is level bounded on $\mathcal{X}$ according to the Assumption \ref{assumption 3}, we easily have that the sequence $\left\{ x^k\right\}$ is bounded for any $k\in \mathbb{N}$.\qed
\end{proof}
\begin{corollary}\label{corollary 2}
 When  $\beta_k$ is chosen as (\ref{100}),  it holds that
 \[ \sum_{k=0}^{+\infty}  {\Vert x^{k+1}-x^k\Vert}^2< +\infty,\]  
 and
 \[\sum_{k=0}^{+\infty} \mu^{-1}_{k} {\Vert x^{k+1}-x^k\Vert}^2<+\infty,\quad
 \sum_{k=0}^{+\infty} \mu^{-1}_{k+1} {\Vert x^{k+1}-x^k\Vert}^2<+\infty.   \]
   
 \end{corollary}
\begin{proof}
 Summing the inequality (\ref{3.11}) from $k=0$ to $k=K$, we get
 \begin{align*}
&\sum_{k=0}^{k=K}(-\tau_k+\frac{1}{2}L\mu^{-1}_k){\Vert x^{k+1}-x^k\Vert}^2 \notag\\
 \leq & H (x^{0},x^{-1} ,\mu_{-1})
-H (x^{K+1},x^K,\mu_K)+\kappa(\mu_{-1}-\mu_K),
\end{align*}
that is
 \begin{align*}
&\sum_{k=0}^{k=K}\frac{L}{4}(\mu^{-1}_{k}-\beta^2_k\mu^{-1}_{k+1}){\Vert x^{k+1}-x^{k}\Vert}^2 \notag\\ =&\sum_{k=0}^{k=K}\frac{L}{4}(\mu_{k+1}\mu^{-1}_k-\beta^2_k)\mu^{-1}_{k+1}{\Vert x^{k+1}-x^{k}\Vert}^2  \notag\\
 \leq& H (x^{0},x^{-1} ,\mu_{-1})
-H (x^{K+1},x^K,\mu_K)+\kappa(\mu_{-1}-\mu_K).
\end{align*}
Since $\beta^2_k\frac{\mu_k}{\mu_{k+1}}\leq1-a\frac{\mu_k}{\mu_{k+1}}$, $\mu_k$ is nonincreasing and $\mu_k\in(0,1)$ for any $k\in\mathbb{N}$,
it holds that
\[ \beta^2_k\frac{\mu_k}{\mu_{k+1}}\leq1-a,
\]
thus,
\[ \mu^{-1}_k-\beta^2_k\mu^{-1}_{k+1}\ge a,\quad\mu_{k+1}\mu^{-1}_k-\beta^2_k\ge a.
\]
When $K$ is large enough, since $\mu_k$ is nonincreasing, it holds that
\begin{align*}
\sum_{k=0}^{+\infty}{\Vert x^{k+1}-x^{k}\Vert}^2 
&\leq  H (x^{0},x^{-1} ,\mu_{-1})+\kappa\mu_{-1}-\zeta  <+\infty,
\end{align*}
\begin{align*}
\sum_{k=0}^{+\infty}\mu^{-1}_{k}{\Vert x^{k+1}-x^{k}\Vert}^2 
&\leq\sum_{k=0}^{+\infty}\mu^{-1}_{k+1}{\Vert x^{k+1}-x^{k}\Vert}^2 \notag\\
&\leq  H (x^{0},x^{-1} ,\mu_{-1})+\kappa\mu_{-1}-\zeta  <+\infty.
\end{align*}
The proof is completed .\qed
\end{proof}
Next, the following theorem is given to illustrate the convergence of the subsequence.
 \begin{theorem}\label{thm1}
 Suppose $\beta_k$ is chosen as (\ref{100}), any accumulation point of $ \left\{x^k\!:\! k\in\mathcal{N}^s \right\}$ is a lifted stationary point of (\ref{1.2}).
  \end{theorem}
\begin{proof}

  According to Corollary \ref{corollary 1}, it follows that $\left\{ x^k: k\in\mathcal{N}^s \right\}$ is bounded, so it  has an accumulation point which might as well denote as $\overline x$, i.e., there exists a subsequence $\left\{ x^{k_i} \right\}_{k_i\in\mathcal{N}^s}$ of $\left\{ x^k \right\}_{k\in\mathcal{N}^s}$ such that $x^{k_i}\to \overline x$ as $i\to +\infty$.
 For any $k\in\mathcal{N}^s$, the inequality (\ref{3.9}) does not hold. Together this with (\ref{3.11}), we get
$ (-\tau_{k_i}+\frac{1}{2}L\mu^{-1}_{k_i}){\Vert x^{k_i+1}-x^{k_i}\Vert}^2\leq \alpha{\mu^2_{k_i}} ,$ that is,
\[
 \frac{L}{4}(1-\beta^2_{k_i}\frac{\mu_{k_i}}{\mu_{k_i+1}})\mu^{-1}_{k_i}\Vert x^{k_i+1}-x^{k_i}\Vert^2\leq \alpha{\mu^2_{k_i}},
\]
\[
\frac{L}{4}(\mu^{-1}_{k_i}{\mu_{k_i+1}}-\beta^2_{k_i})\mu^{-1}_{k_i+1}\Vert x^{k_i+1}-x^{k_i}\Vert^2\leq \alpha{\mu^2_{k_i}}.
\]
Multiplying $\mu^{-1}_{k_i}$ and $\mu^{-1}_{k_i+1}$ on both sides of the above inequalities respectively, it holds that
\begin{align}
\frac{L}{4}(1-\beta^2_{k_i}\frac{\mu_{k_i}}{\mu_{k_i+1}})\mu^{-2}_{k_i}\Vert x^{k_i+1}-x^{k_i}\Vert^2\leq \alpha{\mu_{k_i}},\label{3.188}\\
\frac{L}{4}(\mu^{-1}_{k_i}{\mu_{k_i+1}}-\beta^2_{k_i})\mu^{-2}_{k_i+1}\Vert x^{k_i+1}-x^{k_i}\Vert^2\leq \alpha{\mu^2_{k_i}}\mu^{-1}_{k_i+1}.\label{3.189}
\end{align}
 Since  $\beta^2_{k_i}{\frac{\mu_{k_i}}{\mu_{k_i+1}}}\leq 1-a\frac{\mu_{k_i}}{\mu_{k_i+1}}$, $\mu_{k_i}$ is nonincreaing and $\mu_{k_i}\in(0,1)$,  it holds that  
  $\beta^2_{k_i}{\frac{\mu_{k_i}}{\mu_{k_i+1}}}\leq 1-a\leq 1-a\mu_{k_i}$.Then
  \[  1-\beta^2_{k_i}\frac{\mu_{k_i}}{\mu_{k_i+1}}\ge a,\quad \mu^{-1}_{k_i}{\mu_{k_i+1}}-\beta^2_{k_i}\ge a \]
   
   Moreover,  $\lim_{i\to +\infty} \mu_{k_i}=0$.   
 Taking  limits on (\ref{3.188}), (\ref{3.189}), the right hand of them are tends to $0$, then we get   
\begin{align*}
\lim_{i\to +\infty} \mu^{-1}_{k_i} \Vert x^{k_i+1}-x^{k_i}\Vert =0,\quad
\lim_{i\to +\infty} \mu^{-1}_{k_{i+1}} \Vert x^{k_i+1}-x^{k_i}\Vert =0,
\end{align*}
and then
\begin{align}\label{3.18a}
\lim_{i\to+\infty}\mu^{-1}_{k_i}(x^{k_i+1}-y^{k_i})
=\lim_{i\to+\infty}\mu^{-1}_{k_i}((
x^{k_i+1}-x^{k_i})-\beta_{k_i}(
x^{k_i}-x^{k_i-1}))=
0. 
\end{align}
By the first order necessary optimality condition of (\ref{3.7a}), we have
   \begin{align}\label{3.18}
\langle \nabla\widetilde f(x^{k_i}, \mu_{k_i})+L\mu^{-1}_{k_i}(x^{k_i+1}-y^{k_i})+\lambda\xi^{k_i}, x-x^{k_i+1}\rangle\ge 0, 
\end{align}
where $\xi^{k_i}\!\in\!\partial\Phi^{d^{k_i}}(x^{k_i+1}), x\in\mathcal{X}.$
Due to $\lim_{i\to+\infty}x^{k_i}=\overline x$,  $\lim_{i\to+\infty} \Vert x^{k_{i}+1}-x^{k_i} \Vert=0$,  then $\lim_{i\to+\infty}x^{k_i+1}=\overline x $. Since
the finiteness of the set $ \left\{d^{k_i}: i\in\mathbb{N}\right\}$,  there is a subsequence $ \left\{k_{i_j} \right\}$
of $ \left\{ {k_i} \right\}$ and 
 $\overline{d}\in \mathcal{D}(\overline x)$ such that $d^{k_{i_j}}=\overline d$,  $\forall j\,\in\mathbb{N}$.
According to the upper semicontinuity of $\partial\Phi^{\overline d}$ and  $\lim_{j\to+\infty}x^{k_{i_j}}=\lim_{j\to+\infty}y^{k_{i_j}}=\overline x $, we obtain
\begin{align}\label{3.19}
\left\{ \lim_{j\to+\infty}\xi^{k_{i_j}}: \xi^{k_{i_j}}\in\partial\Phi^{d^{k_{i_j}}}(x^{k_{i_j}+1}) \right\}\subseteq \partial\Phi^{\overline d}(\overline x).
\end{align}
Combining (\ref{3.18a}), (\ref{3.18}), (\ref{3.19}) with the Definition \ref{def3.1}(iii), letting $k_i= k_{i_j}$
in (\ref{3.18}) and  $j\to+\infty$, it deduces that there exist $\overline \psi\in \partial f(\overline x)$ and
$\overline \xi^{\overline d} \in\partial\Phi^{\overline d}(\overline x)$ satisfying the following
\begin{align}\label{3.20}
 \langle \overline \psi+\lambda\overline \xi^{\overline d},  x-\overline x\rangle \ge0, \,\forall\, x\in\mathcal{X}. 
\end{align}    
Thus, we have that $\overline x$ is a lifted stationary point of (\ref{1.2}) by $\overline d\in \mathcal{D}(\overline x)$, the definition of $\Phi^{\overline d}$ and the convexity of $\mathcal{X}$.
Then the proof is completed.\qed
\end{proof}
 	\begin{theorem}\label{lemma3.4}
 	Suppose  $\beta_k$ is chosen as in (\ref{100}), it has that:
 	\begin{enumerate}[(i)]
 		\rm\item for any $k\in \mathbb {N}$, $\mathcal{A}(x^{k+1})\subseteq \mathcal{A}(x^{k})\cup\mathcal{A}(x^{k-1}), $\\
 		where $ \mathcal{A}(x) =\left\{ i=\{1,2,\cdots,n\}:x_i\ne0 \right\};$
 		\rm\item  $\Vert\nabla\widetilde f(y^k,\mu_{k-1})\Vert <\frac{1}{2}(\lambda/v+L_f)$ when $k$ is sufficiently large;
 		\rm\item for any $k\in\mathbb N$,
 		there exists $K\in\mathbb{N}$ such that for all $k\ge K$
 		\hspace*{1em} \begin{enumerate}
	\item   $k\in\mathcal{M}_{1}$ or   $k\in\mathcal{M}_{2}$ or $k\in\mathcal{M}_{3}$;\\
 	where $\mathcal{M}_{1}=\{k:for\,each\, i=1,\cdots,n,\,there\, exist\, two\, consecutive\,\\ sequence \,values\, have \,values\, of\, 0\};$\\
 $\mathcal{M}_{2}=\{k: \Vert x^{k+1}-x^k\Vert+  \Vert x^{k}-x^{k-1}\Vert\ge C\mu_k\};$\\
 $\mathcal{M}_{3}=\{k: \Vert x^{k+2}-x^{k+1}\Vert+  \Vert x^{k+1}-x^{k}\Vert\ge C\mu_{k+1}\}.$\\
 where $C= \min\{\frac{L^{-1}}{{2c}},\frac{3}{4} L^{-1}(\lambda/v-L_f) \}.$
 \hspace*{1em}   \item  
 $\sum_{k}\Vert x^{k+1}-x^k\Vert<+\infty$, and we have $x^k\to x^*$ as $k\to+\infty$,
 where $x^*$ is a lifted stationary  point of (\ref{1.2}).
 \end{enumerate} 		
 	\end{enumerate}
 \end{theorem}
 \begin{proof}
 	(i) We just only need to prove that if $x^{k-1}_i=x^k_i=0$, we have $ x^{k+1}_i=0$.
 	Suppose $x^{k-1}_i=x^k_i=0$, then $d^k_i=1$   and $ y^k_i=x^k_i+\beta_{k-1}(x^k_i-x^{k-1}_i)=0. $
 	According to (\ref{3.6}) and $ v<\frac{\lambda}{L_f}$, we have\\
 $   \vert y^k_i-L^{-1}\mu_k\nabla_i\widetilde f(y^k,\mu_k)\vert\leq  L^{-1}\mu_k\Vert \nabla\widetilde f(y^k,\mu_k)\Vert < (\lambda L^{-1}\mu_k )/v. $\\
 	Thus, $ x^{k+1}_i=0 $ from (\ref{3.6}), so the result is proved.\\
 	(ii) The verification is similar to the Lemma 3.6 in \cite{bianSmoothingProximalGradient2020}, it is easy to get the result.\\
 	(iii) (a) Letting $ w^k=y^k- L^{-1}\mu_k\nabla\widetilde f(y^k,\mu_k)$, for any  $k\ge K$,  $j\in \left\{1,2,\cdots,n \right\} $, either $ \vert x^k_j\vert<v$ or 
 	$ \vert x^k_j\vert\ge v$. 
 	
 	First we disuss the index $j$ with respect to  $0<\vert x^k_j \vert <v$ on different three cases.\\
 	Case 1. If $\vert w^k_j \vert\leq \frac{\lambda}{v}L^{-1}\mu_k$, then we have $x^{k+1}_j=0$, which means that $d^{k+1}_{j} =1$ from (\ref{2.3}). We also discuss it on three different cases. 
 	 	 	
 	Case (1): If $\vert w^{k+1}_j \vert\leq \frac{\lambda}{v}L^{-1}\mu_{k+1}$, then $x^{k+2}_j=0$. 
 	From $\mathcal{A}(x^{k+3})\subseteq \mathcal{A}(x^{k+2})\cup\mathcal{A}(x^{k+1}), $ we get that
 	$x^{k+3}_j \equiv 0$ for sufficiently large $k$.  
 	
 	Case (2): If  $ w^{k+1}_j>\frac{\lambda}{v}L^{-1}\mu_{k+1}$, we have
 	\begin{align} \label{3.21a}
 		\vert x^{k+2}_j\vert =\vert x^{k+2}_j-x^{k+1}_j\vert\ge & \vert x^{k+2}_j-y^{k+1}_j\vert -\vert y^{k+1}_j-x^{k+1}_j\vert  \notag\\
 		= &  \vert x^{k+2}_j-y^{k+1}_j\vert -\beta_{k}\vert x^{k}_j\vert  \notag\\
 		\ge& \frac{\lambda}{v}L^{-1}\mu_{k+1}-\vert L^{-1}\mu_{k+1}\nabla_j\widetilde f(y^{k+1},\mu_{k+1})\vert-\vert x^{k}_j\vert \notag\\
 		\ge & \frac{1}{2}(\frac{\lambda}{v}-L_f)L^{-1}\mu_{k+1}-\vert x^{k}_j\vert,
 	\end{align}
 	where the second inequality derives from (\ref{3.6}), $0\leq\beta_{k}<1$
 	and the third inequality derives from (ii) of this Lemma.
 	
 	Case(3): 
 	If $ w^{k+1}_j<-\frac{\lambda}{v}L^{-1}\mu_{k+1}$, similar to the process of Case 1(2), it holds that
 	\begin{align}\label{3.211a}
 		\vert x^{k+2}_j\vert
 		\ge  \frac{1}{2}(\frac{\lambda}{v}-L_f)L^{-1}\mu_{k+1}-\vert x^{k}_j\vert,
 	\end{align}  
 Therefore, from (\ref{3.21a}) and (\ref{3.211a}), we have that
 \begin{align*}
 &\Vert x^{k+2} -x^{k+1} \Vert^2+\Vert x^{k+1} -x^k\Vert^2\\
 \ge&\vert x^{k+2}_j-0\vert^2 +\vert x^k_j-0\vert^2\\
 \ge& \frac{1}{8}(\frac{\lambda}{v}-L_f)^2(L^{-1}\mu_{k+1})^2,
 \end{align*}
where the second inequality derives from $ 2(a^2+b^2)\ge (a+b)^2$ for any $a,b\ge0$.
	According to $2(a+b)^2\ge a^2+b^2 $ for any $a,b\ge0$ again, we have
\begin{align}\label{3.02}
\Vert x^{k+2}-x^{k+1}\Vert+  \Vert x^{k+1}-x^{k}\Vert\ge \frac{1}{4}(\frac{\lambda}{v}-L_f)L^{-1}\mu_{k+1}.
\end{align}
 Case 2. If $ w^k_j>\frac{\lambda}{v}L^{-1}\mu_k$,  
 	similar to the process of Case 1(2), it holds that 
 	\begin{align} \label{3.21}
 		\vert x^{k+1}_j-x^k_j\vert 
 		\ge  \frac{1}{2}(\frac{\lambda}{v}-L_f)L^{-1}\mu_k-\vert x^{k}_j-x^{k-1}_j\vert,
 	\end{align}	
 	Case 3. If $ w^k_j<-\frac{\lambda}{v}L^{-1}\mu_k$, similar to the process of Case 1(3), we have
 	\begin{align}\label{3.22}
 		\vert x^{k+1}_j-x^k_j\vert
 		\ge  \frac{1}{2}(\frac{\lambda}{v}-L_f)L^{-1}\mu_k-\vert x^{k}_j-x^{k-1}_j\vert.
 	\end{align} 	
 	Thus, by using $2(a+b)^2\ge a^2+b^2 $ for any $a,b\ge0$, from (\ref{3.21}) and (\ref{3.22}), we get
 \begin{align}\label{3.03}
 \Vert x^{k+1}-x^k\Vert+  \Vert x^{k}-x^{k-1}\Vert\ge \frac{1}{4}(\frac{\lambda}{v}-L_f)L^{-1}\mu_k.
 \end{align}
 
  	Next, we focus on the index $j$ satisfing  $ x^k_j =0$, which means that $d^k_j=1$. Analogously, from the analysis of Case 1(2) and Case 1(3) of  $0<\vert x^k_j \vert< v$, we easily have $x^{k+1}_j=0$ or $\Vert x^{k+1}-x^{k}\Vert+  \Vert x^{k}-x^{k-1}\Vert\ge \frac{1}{4}(\frac{\lambda}{v}-L_f)L^{-1}\mu_{k}.$ 
 		
 At last, we disuss the index $j$  with respect to  $\vert x^k_j\vert \ge v$, which means that $d^k_j=2$ or $3$. Due to the analysis about  $d^k_j=2$ or $3$ is similar, we just discuss it when $d^k_j=2$.\\	
 	Case 1. If $\vert w^k_j+ \frac{\lambda}{v}L^{-1}\mu_k\vert\leq \frac{\lambda}{v}L^{-1}\mu_k$, then we have $x^{k+1}_j=0$, which means that $d^{k+1}_{j} =1$. We also discuss it on three different cases. 
 	
 	Case (1): If $\vert w^{k+1}_j \vert\leq \frac{\lambda}{v}L^{-1}\mu_{k+1}$, then $x^{k+2}_j=0$. 
 	From $\mathcal{A}(x^{k+3})\subseteq \mathcal{A}(x^{k+2})\cup\mathcal{A}(x^{k+1}), $ we get that
 	$x^{k+3}_j\equiv 0$ for sufficiently large $k$.  
 		
 	Case(2): If  $ w^{k+1}_j>\frac{\lambda}{v}L^{-1}\mu_{k+1}$, similar to the analysis of Case 1(2)  of  $0<\vert x^k_j \vert< v$,
 	we have the result (\ref{3.21a}). 	
 	
 	Case(3): 
 	If $ w^{k+1}_j<-\frac{\lambda}{v}L^{-1}\mu_{k+1}$, similar to the analysis of Case 1(3)  of  $0<\vert x^k_j \vert< v$,  we have the result (\ref{3.211a}).\\
 	Therefore, for case 1(2) and case 1(3) with respect to  $\vert x^k_j\vert \ge v$, we have (\ref{3.02}).\\
 	Case 2. If $ w^k_j+\frac{\lambda}{v}L^{-1}\mu_k>\frac{\lambda}{v}L^{-1}\mu_k$, 
 	\begin{align} 
 		\vert x^{k+1}_j-x^{k}_j\vert \ge & \vert x^{k+1}_j-y^{k}_j\vert -\vert y^{k}_j-x^{k}_j\vert  \notag\\
 		= &  \vert x^{k+1}_j-y^{k}_j\vert -\beta_{k-1}\vert x^{k}_j-x^{k-1}_j\vert  \notag\\
 		\ge& \vert L^{-1}\mu_{k}\nabla\widetilde f(y^{k},\mu_{k})\vert-\vert x^{k}_j-x^{k-1}_j\vert \notag  
 	\end{align} 
 	it holds that 
 	\begin{align}  
 		\vert x^{k+1}_j-x^k_j\vert 
 		\ge   L^{-1}\mu_{k}\vert\nabla_j\widetilde f(y^{k},\mu_{k})\vert-\vert x^{k}_j-x^{k-1}_j\vert \notag
 	\end{align}
 	Since $\widetilde f(y^{k},\mu_{k})$ is gradient lipschitz continuous on $3\mathcal {X}$  with respect to $y^{k}$, there exist a constant $c>0$ such that
 	$	\vert\nabla_j\widetilde f(y^{k},\mu_{k})\vert\ge c$, that is
 	\begin{align}\label{3.06}  
 		\vert x^{k+1}_j-x^k_j\vert 
 		\ge  cL^{-1}\mu_{k+1}-\vert x^{k}_j-x^{k-1}_j\vert  
 	\end{align} 	
 	Case 3. If $ w^k_j+\frac{\lambda}{v}L^{-1}\mu_k<-\frac{\lambda}{v}L^{-1}\mu_k$,  it has that
 	\begin{align} \label{3.07}
 		\vert x^{k+1}_j-x^{k}_j\vert \ge & \vert x^{k+1}_j-y^{k}_j\vert -\vert y^{k}_j-x^{k}_j\vert  \notag\\
 		\ge& \frac{2\lambda L^{-1}\mu_k}{v}-\vert L^{-1}\mu_{k}\nabla\widetilde f(y^{k},\mu_{k})\vert
 		-\vert x^{k}_j-x^{k-1}_j\vert \notag  \\
 		\ge& \frac{2\lambda L^{-1}\mu_k}{v}- \frac{L^{-1}\mu_{k}}{2}(\frac{\lambda}{v}+L_f)
 		-\vert x^{k}_j-x^{k-1}_j\vert \notag  \\
 		\ge& \frac{3}{2}(\frac{\lambda}{v}-L_f)L^{-1}\mu_{k} -\vert x^{k}_j-x^{k-1}_j\vert.
 	\end{align}
  Therefore, from (\ref{3.06}) and (\ref{3.07}), we have
 	\begin{align}\label{3.09}
 	\Vert x^{k+1}-x^{k}\Vert+  \Vert x^{k}-x^{k-1}\Vert\ge C\mu_{k}.
 	\end{align}
 	where $C = \min \{\frac{3}{4}(\frac{\lambda}{v}-L_f)L^{-1},\frac{c}{2}L^{-1}\}$.
 	
 In conclusion, for any $i=\{1,2,\cdots,n \}$ and sufficiently large $k$, we have the result
  $x^{k}_i=0, x^{k+1}_i=0$ or $ \Vert x^{k+1}-x^k\Vert+  \Vert x^{k}-x^{k-1}\Vert\ge C\mu_k$ or $ \Vert x^{k+2}-x^{k+1}\Vert+  \Vert x^{k+1}-x^{k}\Vert\ge C\mu_{k+1}$. That is, $k\in\mathcal{M}_1$ or $k\in\mathcal{M}_2$ or $k\in\mathcal{M}_3$. It is need to mention that $\mathcal{M}_2$ and $\mathcal{M}_3$ may have the same elements.\\
 	This completes the proof.\\
 	(b)  Since		
 	$\mathcal{A}(x^{k+2})\subseteq\mathcal{A}(x^{k+1}) \cup \mathcal{A}(x^{k})$,
 	for  $k\in\mathcal{M}_{1} $, it  satisfies $x^{k+2}\equiv 0$, therefore,
 	$\sum_{k\in\mathcal{M}_{1}}\Vert x^{k+1}-x^{k}\Vert < +\infty.$
 	
 	For $k\in\mathcal{M}_{2}$, it follows that
 	\begin{align}
 		& C\sum_{k\in\mathcal{M}_{2}}\Vert x^{k+1}-x^{k}\Vert +\Vert x^{k}-x^{k-1}\Vert \notag\\
 		\leq  &C\sum_{k\in\mathcal{M}_{2}}\mu^{-1}_k (\Vert x^{k+1}-x^{k}\Vert +\Vert x^{k}-x^{k-1}\Vert)^2  \notag\\
 		\leq& C\sum_{k\in\mathcal{M}_{2}}\mu^{-1}_k
 		(2{\Vert x^{k+1}-x^k\Vert}^2 +2{\Vert x^{k}-x^{k-1}\Vert}^2)\notag\\
 		\leq& C\sum_{k\in\mathcal{M}_{2}}
 		(2{\mu^{-1}_{k+1}\Vert x^{k+1}-x^k\Vert}^2 +2\mu^{-1}_{k}{\Vert x^{k}-x^{k-1}\Vert}^2)\notag\\
 		< &+\infty,  \notag
 	\end{align}
 	where the first inequality comes from the result (iii)(a) of this Lemma, the second one derives from $(a+b)^2\leq 2(a^2+b^2)$ for any $a,b\ge 0$, the third one due to the nonincreasing of $\mu_k$ and the last one comes from Corollary \ref{corollary 2}.
 	From the above analysis, then we get
 	\begin{align}\label{3.26}
 		C\sum_{k\in\mathcal{M}_{2}}\Vert x^{k+1}-x^{k}\Vert
 		&\leq C\sum_{k\in\mathcal{M}_{2}}(\Vert x^{k+1}-x^{k}\Vert +\Vert x^{k}-x^{k-1}\Vert) \notag\\
 		&< +\infty,
 	\end{align}
 	that is,
 	\[
 	\sum_{k\in\mathcal{M}_{2}}\Vert x^{k+1}-x^{k}\Vert < +\infty.
 	\]
 	Similarly, \[
 	\sum_{k\in\mathcal{M}_{3}}\Vert x^{k+1}-x^{k}\Vert \ge +\infty.
 	\]
Thus,
	\begin{align*}
	&\sum_{k}\Vert x^{k+1}-x^k\Vert\\
	\leq& \sum_{k\in\mathcal{M}_{1}}\Vert x^{k+1}-x^k\Vert+\sum_{k\in\mathcal{M}_{2}}\Vert x^{k+1}-x^k\Vert+\sum_{k\in\mathcal{M}_{3}}\Vert x^{k+1}-x^k\Vert
 	<+\infty. 
 	\end{align*}
 	Moreover, since $\{x^k\}$ is bounded, then we have $x^k\to x^*$ as $k\to+\infty$, where $x^*$ is a lifited stationary point of (\ref{1.2}).\qed
 	\begin{remark}
 In fact,
 the global convergence of the SPG algorithm in \cite{bianSmoothingProximalGradient2020} can also be obtained by the idea of discussing each case according to the index of sequence value. 		
 		\end{remark}
 \end{proof}

%

%
At the last part of the section, we aim at proving the local convergence rate of SPGE with respect to the proximal residual,  where the proximal residual is defined as
\begin{align}\label{3.27}
r(x):= x- prox_g(x-\nabla f(x)),
\end{align}
where $g$ is convex (possibly nonsmooth) function and $f$ is a smooth(possibly nonconvex) function. Obviously, $r(x)=0$ is equivalent to $x=prox_g(x-\nabla f(x))$ from \cite{parikhProximalAlgorithms2014b}, which characterize  the first order optimality condition of $\min_{x\in R^n} f(x)+g(x)$. Thus, it can also be seen as a quality to measure
the convergence of the sequence  $\left\{ x^k \right\}$ to a critical point. First, we present some preparatory work.
\begin{lemma}\rm{(\cite{ochsIPianoInertialProximal2014b,Nesterov Y. Gradient methods})}
\label{lemma3.5}
Suppose $y,z\in R^n$, $\alpha>0$. Define the functions
\[
p_g(\alpha):=\frac{1}{\alpha}\Vert prox_{\alpha g}(y-\alpha z)-y\Vert
\]
and
\[ q_g(\alpha):=\Vert prox_{\alpha g}(y-\alpha z)-y \Vert.
\]
Then $p_g(\alpha)$ is a decreasing function of $\alpha$, and $q_g(\alpha)$ is increasing in $\alpha$.
\end{lemma}

\begin{theorem}\label{thm4}
 Suppose that $\left\{ x^{k} \right\}$ is a sequence generated by SPGE algorithm and $r(x^k)$ is a proximal residual defined by (\ref{3.27}). Then for any $k\in\mathbb{N}$ and $\sigma\in(0,1)$, when
 $\overline\mu<L$,
 it holds that
 \[   {\rm\mathop{min}}_{0\le k\le K}{\Vert r(x^k) \Vert}^2\le \frac{C_1}{(K+1)^{1-\sigma}},   \]
    where  $C_1=L^2C_0$ and $C_0=8(H(x^{0},x^{-1} ,\mu_{-1})-\zeta)$ .
  \end{theorem}

 \begin{proof}
       First, we reformulate $Q_{d^k}(x,y^k,\mu_k)$ in the form of proximal type, that is
    \[ {argmin}_{x\in\mathcal X} Q_{d^k}(x,y^k,\mu_k)=
        prox_{ L^{-1}\mu_k\lambda{\Phi}^{d^k}}(y^k- L^{-1}\mu_k\nabla \widetilde f(y^k, \mu_k)).
   \]
   From Lemma \ref{lemma3.5},  we have that
 \begin{align*}
r(x^k)&=\Vert prox_{\lambda{\Phi}^{d^k}}(x^k-\nabla \widetilde f(x^k, \mu_k))-x^k\Vert\\
&\le\frac{1}{\min\{1,L^{-1}\mu_k \}}
\Vert prox_{L^{-1}\mu_k\lambda{\Phi}^{d^k}}(x^k-L^{-1}\mu_k \nabla \widetilde f(x^k, \mu_k))-x^k \Vert,
\end{align*}
Since $\mu_k\leq\overline\mu<L$, it has that
\begin{align*}
 L^{-1}\mu_kr(x^k)&\le
 \Vert prox_{\lambda{\Phi}^{d^k}L^{-1}\mu_k}(x^k-L^{-1}\mu_k \nabla \widetilde f(x^k, \mu_k))-x^k   \Vert\\
 &\le \Vert prox_{\lambda{\Phi}^{d^k}L^{-1}\mu_k}(x^k-L^{-1}\mu_k\nabla \widetilde f(x^k, \mu_k))\\
 &\quad-prox_{\lambda{\Phi}^{d^k}L^{-1}\mu_k}(y^k-L^{-1}\mu_k\nabla \widetilde f(y^k, \mu_k))  \Vert\\
 &\quad+\Vert prox_{\lambda{\Phi}^{d^k}L^{-1}\mu_k}(y^k-L^{-1}\mu_k \nabla \widetilde f(y^k, \mu_k))-x^k \Vert\\
 &\le \Vert x^k-y^k \Vert +L^{-1}\widetilde L\Vert x^k-y^k \Vert+\Vert x^{k+1}-x^k \Vert \\
 &\le 2\Vert x^k-y^{k} \Vert+\Vert x^{k+1}-x^k \Vert,
 \end{align*}
 where the third inequality derives from the nonexpansiveness of the proximal operator and the lipschitz
  continuous of the gradient $\nabla f(\cdot,\mu_k)$.\\
According to $\beta_{k-1}\in[0,1)$ , the above inequality holds as follows
\begin{align}\label{H}
   L^{-1}{\mu_k} r(x^k)\leq2\Vert x^k-x^{k-1}\Vert+\Vert x^{k+1}-x^k \Vert .
\end{align}
Squaring  above inequality on both sides, using $2(a^2+b^2)\ge(a+b)^2$ for any $a,b\ge 0$ and multiplying
$\mu^{-1}_k$ on both sides, it can be reformulated as the following
 \begin{align}\label{C}
  L^{-2}{\mu_k}  r^2(x^k)\!<\!8{\mu^{-1}_k} \Vert\! x^k\!-\!x^{k-1}\!\Vert^2
\!+\!2\mu^{-1}_k \Vert\! x^{k+1}-x^k\!\Vert^2,
 \end{align}
 Summing  up (\ref{C}) from $k=0$ to $K$, it yields
\begin{align*}
  \sum_{k=0}^{K} L^{-2}{\mu_k}  r^2(x^k)
&\le  \sum_{k=0}^{K}8 {\mu^{-1}_{k}}\Vert x^k-x^{k-1}\Vert^2+\sum_{k=0}^{K}2\mu^{-1}_k \Vert x^{k+1}-x^k \Vert^2,
\end{align*}
Further, from the nonincreasing of $\mu_k$ and from Corollary \ref{corollary 2}, we have
\begin{align*}
&({L^{-2}}{\mu_K} (K+1)){\rm{\mathop{min}\limits_{0\le k\le K}}}{r^2(x^k)}\\
\le&(L^{-2} (K+1)) {\rm{\mathop{min}\limits_{0\le k\le K}}}\mu_k{r^2(x^k)}\\
 \le&  L^{-2} \sum_{k=0}^{K} {\mu_k r^2(x^k) } \\
 \le& \sum_{k=0}^{K} 8{\mu^{-1}_{k}}\Vert x^k-x^{k-1}\Vert^2+\sum_{k=0}^{K}2\mu^{-1}_k \Vert x^{k+1}-x^k \Vert^2\\
\le & C_0,
\end{align*}
 where $C_0=8(H(x^{0},x^{-1} ,\mu_{-1})-\zeta) $. \\
According to $\mu_k\ge \frac{\mu_0}{k^\sigma},$ it has that
\begin{align}\label{B}
 {\rm{\mathop{min}\limits_{0\le k\le K}}}{r^2(x^k)}\le \frac{C_1}{(K+1)\mu_{K+1}} \le\frac{C_1}{(K+1)^{1-\sigma}},
\end{align}
where $C_1=L^2C_0$, $0< \sigma <1$.
This completes the proof. \qed
 \end{proof}

  \begin{remark}
  From Theorem \ref{thm1}, it is observed that the convergence of a subsequence of ${\mu^{-1}_k}\Vert x^{k+1}-x^k\Vert$ is a sufficient optimality condition of the SPGE algorithm. Then from the inequality
  (\ref{H}), that is
    \[L^{-1}r(x^k)<2{\mu^{-1}_k} \Vert x^k-x^{k-1}\Vert+ {\mu^{-1}_k}\Vert x^{k+1}-x^k \Vert ,\]it is concluded that the proximal residual $r(x^k)$ can be taken as a quality to measure the convergence of the SPGE algorithm. 
 \end{remark}
 \section{Numerical Experiments }\label{sec4}
In this section, we perform  numerical experiments on three different problems which are deriving from \cite{bianSmoothingProximalGradient2020} to illustrate the efficiency of the SPGE algorithm  and all the experiments are carried out  on 1.80GHz Core i5 PC with 12GB of RAM. The stopping standard  is
set as
\begin{align}\label{4.1}
\rm number\, of\, iterations\, \leq Maxiter \qquad or \qquad \mu_k\leq\epsilon.
\end{align}

In the following, we use ''Iter'' to present the number of iterations , ''time'' to present the CPU time in seconds,  $\widetilde x$ and  $\overline x$ to present the output of the SPG algorithm and SPGE algorithm respectively. The extrapolation item is $ \beta_k\in[0,\sqrt{(1-a\frac{\mu_{k+1}}{\mu_k})\frac{\mu_{k+1}}{\mu_k}}]$ with $0<a\ll 1$. The value of $a$ is chosen as $a=0.0001$ in the following numerical experient.

Before we show the experiments, we present the smoothing function for $l_1$ loss function and censored regression loss function in the following.

For the $l_1$ loss function,
\[ f(x)=\frac{1}{m}\Vert Ax-b \Vert_1,
\]
where$\;A\in R^{m\times n},\;b\in R^m$.\\
The smoothing function is defined as
\begin{align}
\widetilde f(x,\mu)=\frac{1}{m}\sum_{i=1}^{m}\widetilde\theta(A_ix-b_i, \mu) \quad with \quad
\widetilde\theta(s,\mu)=\begin{cases}
\vert s \vert \;                 &if \,\vert s \vert>\mu,\\
\frac{s^2}{2\mu}+\frac{\mu}{2}\; &if \,\vert s\vert\leq\mu.
\end{cases}
\end{align}

     For the nonsmooth convex censored regression loss function
 \[ f(x)= \frac{1}{pm}\sum_{i=1}^{m}\vert \max\left\{A_ix-c_i, 0 \right\}-b_i \vert^p,
 \]
where $p\in[1,2],A^T_i\in R^n$ and $c_i, b_i\in R$, $i=1,\cdots,m$.\\
When $p=1$, the smoothing function is defined as
\begin{align}
\widetilde f(x,\mu)\!=\!\frac{1}{m}\sum_{i=1}^{m}\widetilde\theta(\widetilde\phi( A_ix,\mu)-b_i,\mu_i)\; with \;
\widetilde\phi(s,\mu)\!=\!\begin{cases}
\max \left\{s,0 \right\}                &if \vert s \vert>\mu,\\
\frac{(s+\mu)^2}{4\mu}\; &if \vert s\vert\leq\mu.
\end{cases}
\end{align}

\begin{example}
We first consider the following problem
\begin{align}\label{4.2}
\min\limits_{0\leq x_1,x_2\leq 1}{\mathcal{F}_{l_0}(x_1,x_2)}:=\vert x_1+x_2-1\vert+\lambda
{\Vert x\Vert}_{0}.
\end{align}
\end{example}
 Let $\mathcal{GM}$ denote the global minimizers of (\ref{4.2}).
  \begin{table}[htb]
\centering
\caption{Numerical results of the SPGE and SPG algorithm for problem (\ref{4.3})($m\ll n$)}
 \resizebox{\textwidth}{!}{
\begin{tabular}{ccccccccc} 
\toprule
 \multirow{2}{*} { $\lambda$ }      & \multirow{2}{*}{ $\mathcal{GM}$ }  & \multirow{2}{*}v & \multirow{2}{*} {$\widetilde x $} & \multirow{2}{*}{$\overline x $ }& \multicolumn{2}{l}{Iter} & \multicolumn{2}{l}{time} \\
  \cmidrule(lr){6-7}   \cmidrule(lr){8-9}
 &  &  & & &SPG& SPGE&SPG& SPGE \\
  \midrule
  	0.7/0.8/0.9 &(1,0),(0,1)&0.4/0.5/0.6&(1,0)/(1,0)/(1,0)&(1,0)/(1,0)/(1,0)& 15/14/13&10/7/7 &0.0336/0.0369/0.0341&0.0302/0.0319/0.0302\\  
			1/1/1&(1,0),(0,1),(0,0)&0.7/0.5/0.3&(0,0)/(1,0)/(0.6,0.4)&(1,0)/(0,0)/(1,0)&13/13/14&7/10/10&0.0314/0.0328/0.0353&0.0311/0.0316/0.0349	\\ 
			1.2/1.3/1.4&(0,0)&0.8/0.9/1 &(0,0)/(0,0)/(0,0)& (0,0)/(0,0)/(0,0) &16/25/23 & 7/8/5&0.0299/0.0296/0.0296
		&0.0288/0.0281/0.0281
			   \\
    \bottomrule
  \end{tabular}}
\end{table}

The parameters in SPG algorithm are set as default values and the parameters of  SPGE algorithm are set as in the following:
\[   L=\sqrt 2,  \alpha=1 ,   \sigma=0.9 ,  \rm Maxiter=10^4 ,  \kappa=\frac{1}{2},\epsilon=10^{-3},  L_f=\sqrt 2. \]

The extrapolation term $\beta_k$ in the SPGE algorithm is chosen as  in the way of (\ref{A}). At the initialize point, we set $\mu_{-1}=\mu_0=\mu_1=0.1$ and $t_{-1}=1$ and  $x_0=(1,0.8)'$ for  SPGE algorithm. In order to compare  the numerical efficiency of the SPG algorithm with SPGE algorithm fairly, we perform 20 experiments in each case. The details of the parameters and the results are presented in Table 1. From Table 1, we can see that there is no influence for different $v$ in SPGE algorithm getting global minimizers. In fact, the inertial term in SPGE algorithm  has the ability to get away from the spurious stationary point sometimes though it is not used to find the global minimizer. This phenomenon has happened in other cases, one can refer to \cite{Bertsekas D P Nonlinear,ochsIPianoInertialProximal2014b}.
   While for SPG algorithm, only when $ v\ge 0.5$, it can find the global minimizers.  Besides, it is observed that the iterations and the time of the SPGE algorithm are less than the SPG algorithm on average.

\begin{example}
 Consider the linear regression problem, i.e.,
\begin{align}\label{4.3}
\min\limits_{\boldsymbol{0}\leq x\leq 10\boldsymbol{1_n}}{\mathcal{F}_{l_0}(x)}:=\frac{1}{m}{\Vert Ax-b\Vert}_1+\lambda
{\Vert x\Vert}_{0},
\end{align}
where $ b\in R^m$ and $A\in R^{m\times n}(m\ll n )$.
\end{example}
 The sensing matrix $A$, obeservation $b$ and  the parameters for SPG algorithm are set as in same way in \cite{bianSmoothingProximalGradient2020}, one can refer to \cite{bianSmoothingProximalGradient2020} in details. The parameters for SPGE algorithm are set as in the following:
\[ L=2,\, \alpha=1,\,\sigma=0.9,\, \rm Maxiter=10^4,\, \kappa=\frac{1}{2}, \, \epsilon=10^{-3}.\]
Setting $\lambda=18.8$, $L_f=\Vert A\Vert_{\infty}$ and $v=\min \left\{\frac{\lambda}{L_f},10\right\}$, $x_0= 1.97*ones(n,1)$. The extrapolation term $\beta_k$ in the SPGE algorithm is chosen as  in the way of (\ref{A}) with the fixed restart systems and we perform a fixed restart every 500 iterations.
 We denote the SPGE algorithm as s-FISTA.    At the initialize point, we set $\mu_{-1}=\mu_0=\mu_1=50$ and $t_{-1}=1$.


\begin{figure}[htpb]
\footnotesize
\centering
\begin{minipage}[t] {0.45 \textwidth}
   \includegraphics[scale=0.25]{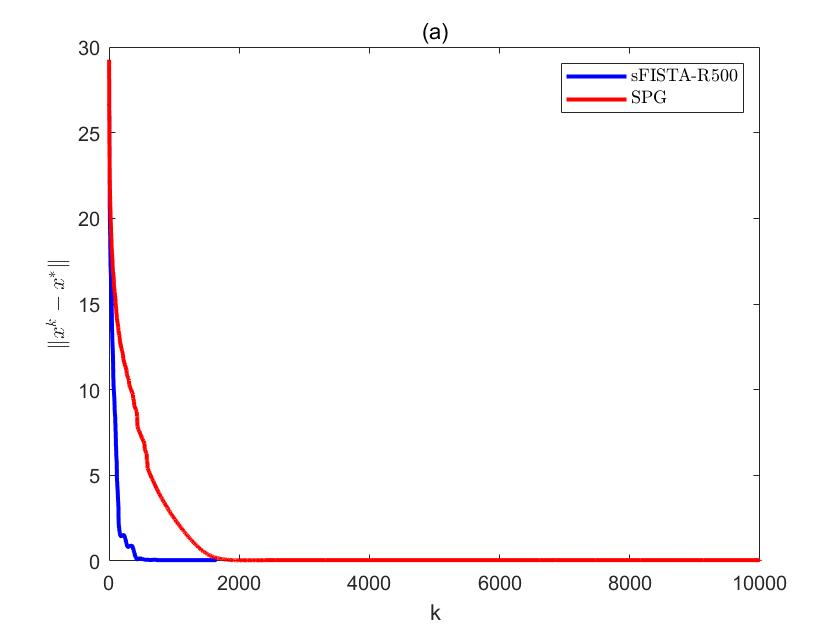}
\end {minipage}
\centering
\begin{minipage}[t] {0.45\textwidth}
	\includegraphics[scale=0.25 ]{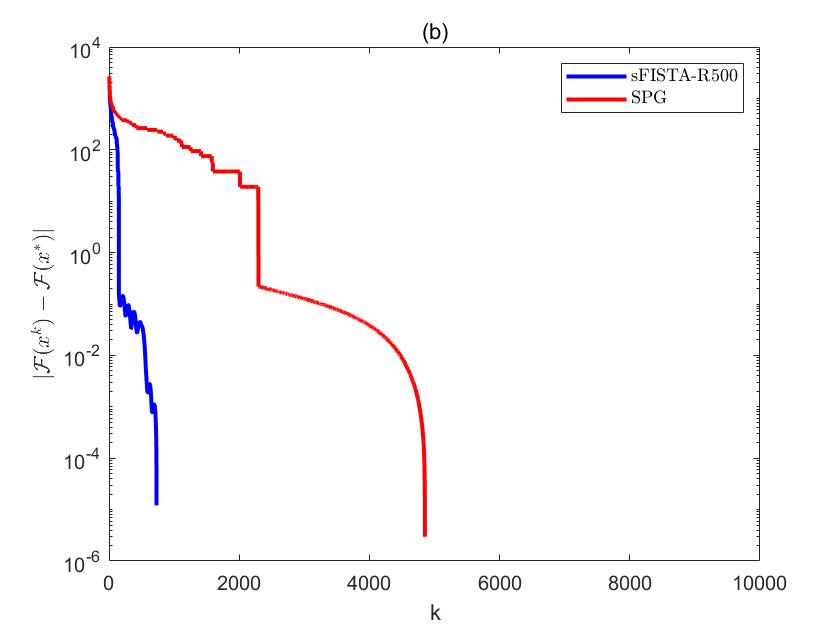}
         \end{minipage}
\centering\caption{$m=60,n=120$ for problem (41)}
  \end{figure}

%
\begin{figure}[htpb]
\footnotesize
\centering
\begin{minipage}[t] {0.45 \textwidth}
   \includegraphics[scale=0.25]{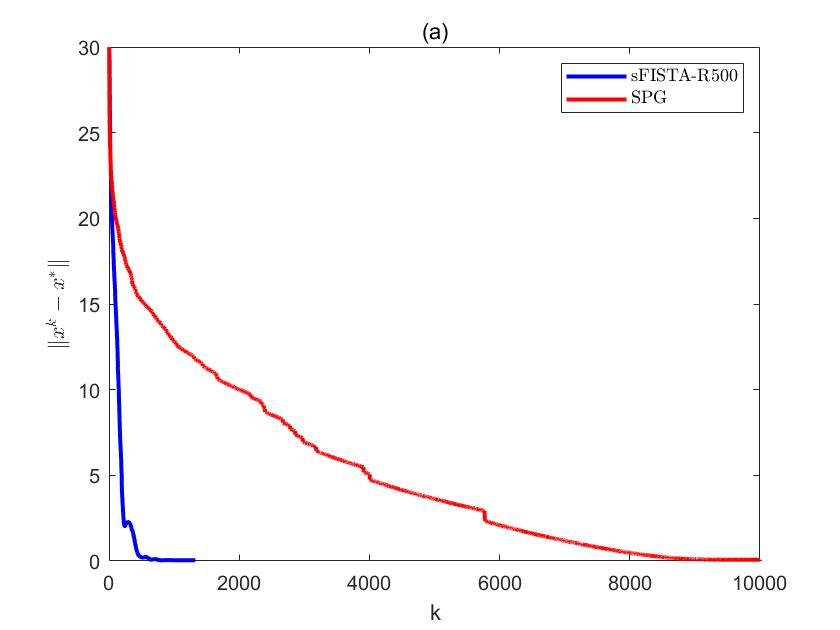}
\end {minipage}
\centering
\begin{minipage}[t] {0.45\textwidth}
	\includegraphics[scale=0.25]{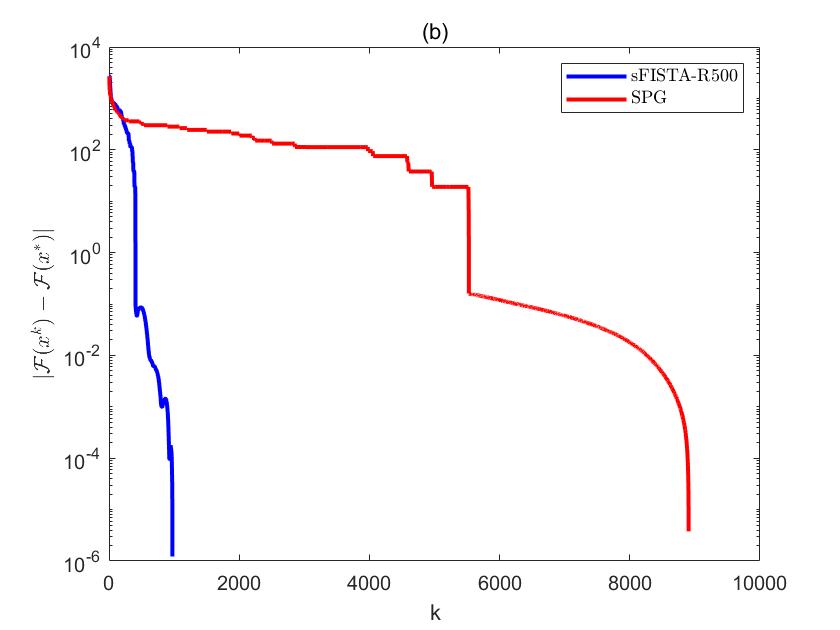}
         \end{minipage}
  \centering\caption{ $m=80,n=160$ for problem (41)}
  \end{figure}

  \begin{figure}[htpb]
\footnotesize
\centering
\begin{minipage}[t] {0.45 \textwidth}
   \includegraphics[scale=0.25]{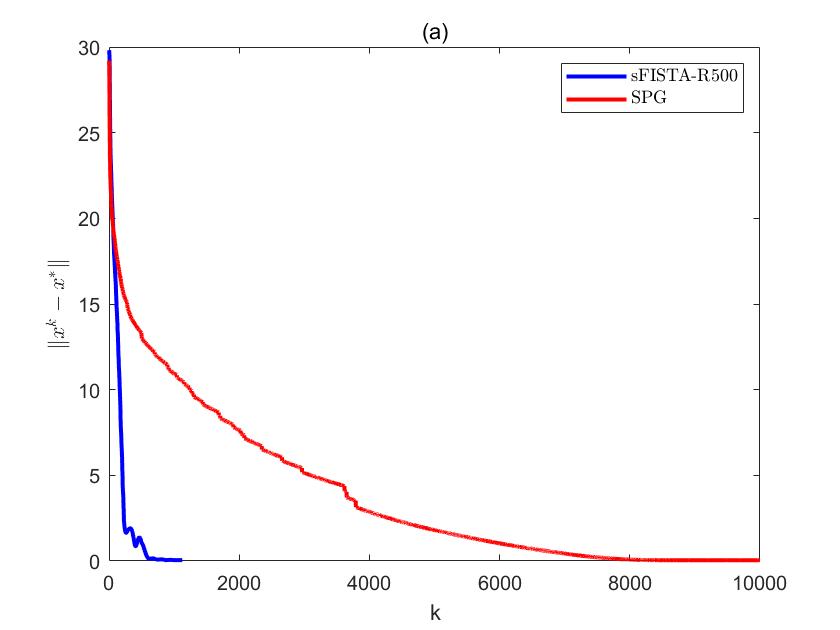}
\end {minipage}
\centering
\begin{minipage}[t] {0.45\textwidth}
	\includegraphics[scale=0.25]{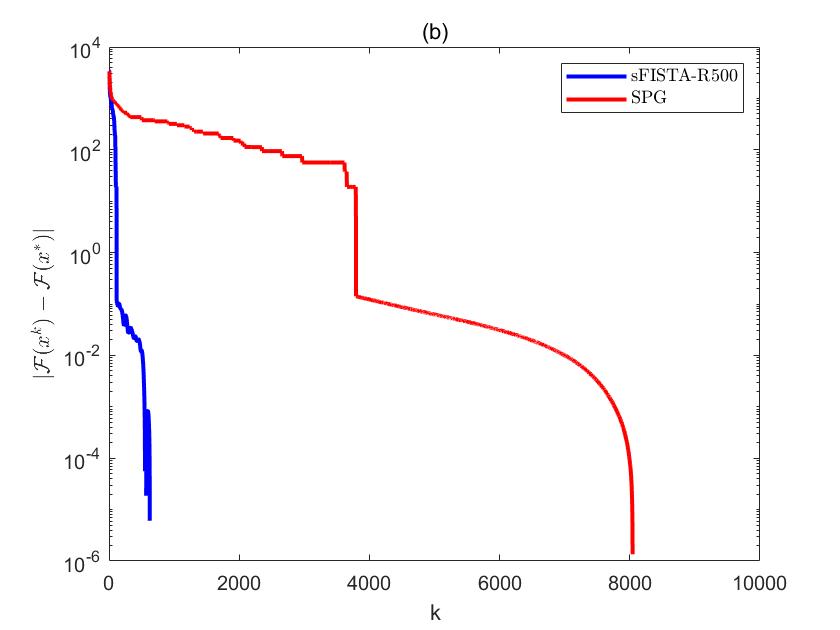}
         \end{minipage}
      \centering\caption{$m=100,n=200$ for problem (41)}
  \end{figure}

First, we present  the effectiveness of  s-FISTA  with fixed restart schemes in  Figure 1, 2, 3 with respect to
$\Vert x^k-x^*\Vert$ and $\vert \mathcal F(x^k)-\mathcal F(x^*)\vert$ against the number of iterations, where $x^*$ is the true solution of the problem (\ref{4.3}). From the part of (a) (b), it is observed that s-FISTA with fixed restart  use fewer iterations as a faster rate than SPG algorithm.
The Figure 4 exhibit the convergence of $\mu_k$  and present that the output of SPGE algorithm
 has the lower bound property and is
very close to original signal $x^*$ with higher probability compared with SPG algorithm.

 \begin{figure}[htpb]
\footnotesize
\centering
\begin{minipage}[t]{0.45\textwidth}
   \includegraphics[scale=0.25]{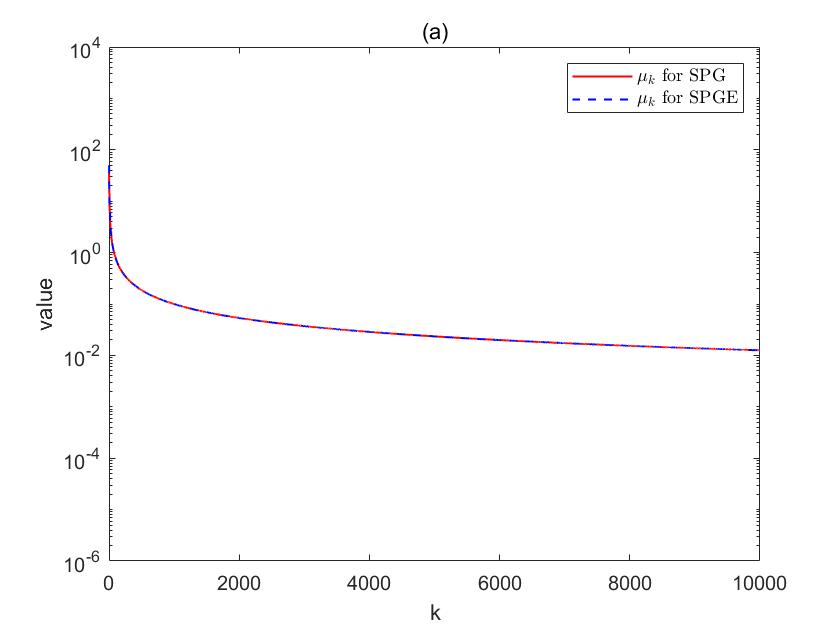}
\end{minipage}
\centering
\begin{minipage}[t]{0.45\textwidth}
	\includegraphics[scale=0.25]{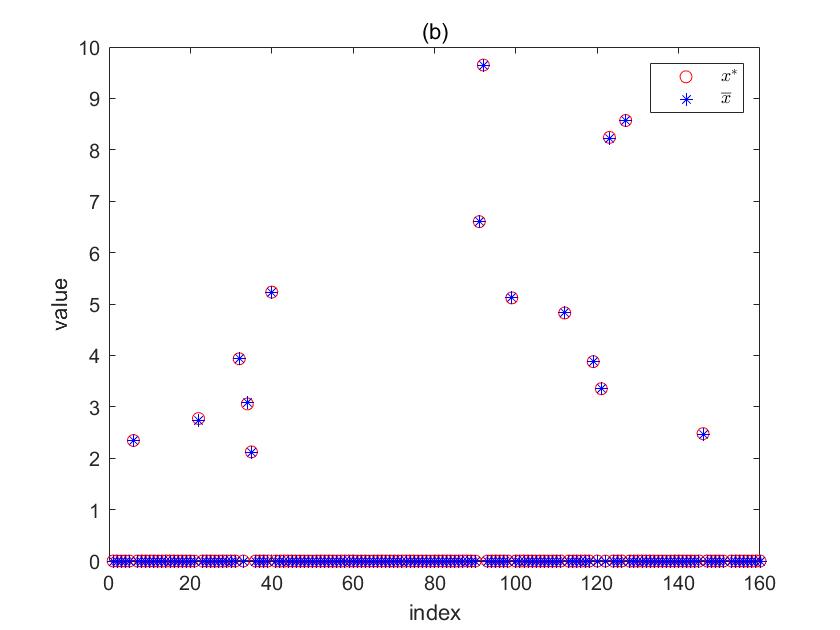}
         \end{minipage}
      \centering\caption{$m=80,n=160$ for problem (41)}
  \end{figure}

Finally, we illustrate the efficiency of the SPGE algorithm from the time, relative error and successful rate aspects. The relative error is defined as in \cite{bianSmoothingProximalGradient2020} and the successful rate is defined as the proportion of the number of elements $\{i:  x^*-\overline x  <= \pm 0.01\textbf{1}_n\}$  in the variable.
  The average results of 20 experiments when $\mathcal A(\overline x)=\mathcal A(x^*)$  are presented in Table 2.
 From Table 2, we view that SPGE algorithm works better than SPG algorithm as a whole and especially we observe that the used time of the SPGE algorithm is far less than the SPG algorithm.

 \begin{table}[htb]
\centering
\caption{Numerical results of the SPGE and SPG algorithm for problem (\ref{4.3})($m\ll n$)}
\begin{tabular}{ccccccccc} 
\toprule
\multirow{2}{*}m &\multirow{2}{*}n & \multirow{2}{*}s & \multicolumn{2}{l}{time}& \multicolumn{2}{l}{rel-err} & \multicolumn{2}{l}{suc-rat} \\
 \cmidrule(lr){4-5}\cmidrule(lr){6-7}\cmidrule(lr){8-9}
 &&  & SPG& SPGE &  SPG& SPGE & SPG& SPGE \\
  \midrule
   60&120&12 &  7.68  &0.96   &0.008   &0.004 &   96$\%$   &95$\%$  \\
  80&160&16 & 8.59    &1.21   &0.006   &0.002 &   96$\%$    &98$\%$    \\
  100 &200&20     &  12.35 &2.19   &0.011   &0.005 &   94.6$\%$  &95.5$\%$ \\
 \bottomrule
  \end{tabular}
\end{table}

 \begin{example}
  Consider censored regression problem, that is,
 \begin{align}\label{4.4}
\min\limits_{\boldsymbol{0}\leq x\leq \boldsymbol{1_n}}{\mathcal{F}_{l_0}(x)}:
=\frac{1}{pm} \sum_{i=1}^{m} \left\{\vert\max \left\{  A_ix-c_i,0 \right\}-b_i   \vert  \right\}+\lambda
{\Vert x\Vert}_{0}.
\end{align}
\end{example}
The goal is to find the true sparse signal $x^*$ about the system $ \max \left\{ Ax^*-c_i,0  \right\} \\ \approx b$ with $A\in R^{m\times n}(m\gg n)$.
The parameters about SPG algorithm are set as default values, one can refer to \cite{bianSmoothingProximalGradient2020} in details and the parameters about SPGE algorithm are set as in the following

$L=1.5$, $\alpha=1$, $ \sigma=0.9$,  $\rm Maxiter=10^4$, $\epsilon=0.01 $ $\kappa=\frac{1}{2}$,
$L_f=\Vert A\Vert_{\infty}$, $\lambda:=\lambda_0L_f$ where $\lambda_0 \in[0.001:0.001:0.1]$.
 The extrapolation term $\beta_k$ in the SPGE algorithm is chosen as  in the way of (\ref{A}) with the fixed restart systems and we perform a fixed restart every 500 iterations.
  At the initialize point, we set $\mu_{-1}=\mu_0=\mu_1=1$ and $t_{-1}=1$.
Let $x^0=0.1*ones(n,1)$, $ v=\min \left\{ \frac{\lambda}{L_f},1 \right\}$.
 We illustrate the efficiency of the SPGE algorithm from the time, relative error, sparsity regression rate and successful rate aspects. The relative error, sparsity regression rate is defined as in \cite{bianSmoothingProximalGradient2020} and the successful rate is defined as in example 4.2.
 The average numerical results of 20 tests are performed to illustrate  the efficiency of SPGE algorithm and
 the details are presented in Table 3. From Table 3, we see that SPGE algorithm  takes much less time than SPG algorithm and  other indexes are better than the SPG algorithm.
\begin{table}[htb]
\centering
\caption{ Numerical results of the SPGE and SPG algorithm for problem (\ref{4.4})($m\gg n$)}
 \resizebox{\textwidth}{!}{
\begin{tabular}{ccccccccccccc} 
\toprule
\multirow{2}{*}m &\multirow{2}{*}n & \multirow{2}{*}s & \multicolumn{2}{l}{time}& \multicolumn{2}{l}{rel-err} & \multicolumn{2}{l}{spa-rat}& \multicolumn{2}{l}{suc-rat}&\multicolumn{2}{l}
{$\mathcal{A}(\overline x)$} \\
 \cmidrule(lr){4-5}\cmidrule(lr){6-7}\cmidrule(lr){8-9}\cmidrule(lr){10-11}\cmidrule(lr){12-13}
 &&  & SPG& SPGE &  SPG& SPGE & SPG& SPGE  & SPG& SPGE& SPG& SPGE \\
  \midrule
      500&100  &20  & 5.42  &1.16  &6.48e-03 &2.43e-03   & 98$\%$   &99$\%$  &97$\%$&99$\%$
         &20.7&19.7\\
     1000&200 & 40 & 30.35 &5.89  &5.48e-03 &2.59e-03   & 98.5$\%$ & 99$\%$  & 90$\%$ &93$\%$&39.97 &39.99\\
    2000&400 & 80 & 236.24 &40.32   &8.00e-03 &5.86e-03  &  98$\%$  &99$\%$ &87$\%$ &90$\%$ & 79.65&79.80 \\
 \bottomrule
  \end{tabular}}
\end{table}
\section{Conclusion }\label{sec4}
This paper mainly propose  the SPGE algorithm  for the exact continuous model of $l_0$ regularization problem \cite{bianSmoothingProximalGradient2020} and discuss its convergence for a fairly general choice of extrapolation parameters, in particular,  the parameters can be chosen as in s-FISTA proposed by Bian \cite{WBA} with fixed restart when $\beta_k\in[0,\sqrt{(1-a\frac{\mu_{k+1}}{\mu_k})\frac{\mu_{k+1}}{\mu_k}}]$ with $a\in(0,1)$ sufficiently small but not equal 0. Especially, it holds that $\sum_{k=0}^{+\infty} {\Vert x^{k+1}-x^k\Vert} <+\infty $ and the whole sequence  $\{x^k\}$ convergences to a lifted stationary point of (\ref{1.2}).
 Further, the rough convergence rate of $o(\frac{1}{k^{1-\sigma}})(\sigma\in(0,1))$  with respect to squared proximal residual is established.
From the numerical experiments, we see that SPGE algorithm is more efficient than SPG algorithm. In particular, in first example, when $v=1$, the SPGE algorithm  can find a global minimizers on different cases. But for SPG algorithm, only $ v\ge 0.5$, it can find the global minimizer. Besides, the iterations and the used time of the SPGE algorithm are less than the SPG algorithm in these examples.

\bibliographystyle {plain}

%

%
%


%

\end{document}